%% file: p2_II.tex
\documentclass[12pt]{amsart}
\usepackage{fullpage,xcolor,graphicx}
\usepackage{hyperref}
\usepackage{pifont}
\usepackage{amsthm}
\usepackage{amsfonts}
\usepackage{amssymb}
\usepackage[mathscr]{euscript}
\usepackage[all]{xy}
\usepackage{amsmath}
\usepackage{epsfig}

\newtheorem{theorem}{Theorem}
\newtheorem{lemma}[theorem]{Lemma}

\theoremstyle{definition}

\theoremstyle{remark}
\newtheorem{remark}{Remark}[section]

\begin{document}
\title{Virtual Covers of Links II}
\author{Micah Chrisman}
\author{Aaron Kaestner}
\begin{abstract} A fibered concordance of knots, introduced by Harer, is a concordance between fibered knots that is well-behaved with respect to the fibrations. We consider semi-fibered concordance of two component ordered links $L=J \sqcup K$ with $J$ fibered. These are concordances that restrict to fibered concordances on the first component. Motivated by some examples of Gompf-Scharlemann-Thompson, we further limit our attention to those links $L$ where $K$ is ``close to'' a fiber of $J$. Such $L$ are studied with virtual covers, where a virtual knot $\upsilon$ is associated to $L$. We show that the concordance class of $\upsilon$ is a semi-fibered concordance invariant. This gives obstructions for certain slice and ribbon discs for the $K$ component. Further applications are to injectivity of satellite operators in semi-fibered concordance and to knots in fibered $3$-manifolds. 
\end{abstract}
\maketitle
In \cite{gst}, Gompf-Scharlemann-Thompson described an infinite family of two component slice links $L_n=J_0 \sqcup V_n$ that are unknown to be ribbon. The knot $J_0$ is a square knot and hence is fibered. Since $L_n$ is concordant to the two component unlink and the unknot is fibered, it is natural to investigate concordances between $J_0$ and the unknot that behave nicely with respect to the fibrations. Such a notion of \emph{fibered concordance} of knots was introduced by Harer \cite{harer}. Thus we consider \emph{semi-fibered concordance} of two component links with first component fibered: concordances restricting to a fibered concordance of the first component.  
\newline
\newline
Scharlemann \cite{scharlemann} later showed how to arrange $V_n$ as a simple closed curve on a fiber of $J_0$. Such links are examples of links in \emph{special Seifert form} (SSF) introduced in \cite{cm_fiber}. These are two component links $L=J \sqcup K$ with $J$ fibered and $K$ lying ``close to'' a fiber (see Section \ref{sec_virt_cov}). Such links may be studied with \emph{virtual covers}, where a virtual knot $\upsilon$ is associated to $L$ so that it functions essentially as an invariant of $L$. In \cite{vc_1} it was shown that $\upsilon$ can detect geometric properties of $L$, such as if $L$ is non-split or non-invertible.
\newline
\newline
These observations provide the setting for this sequel to \cite{vc_1}. The main theorem is that the concordance class of $\upsilon$ as a virtual knot is a semi-fibered concordance invariant of two component links $L$ in SSF (see Section \ref{sec_semi_fib}). An obstruction to the existence of certain kinds of slice and ribbon discs is obtained (see Section \ref{sec_sli_rib}).  We give an example where it is more discriminating than Cochran's $\beta$ invariant \cite{cochran_geom}. The obstruction vanishes for $L_n$ from \cite{gst}.
\newline
\newline
Our second application is to the injectivity of satellite operators. Section \ref{sec_sat} gives a combinatorial condition on $\upsilon$ under which the untwisted satellite operator $J \sqcup K \to J \sqcup P(K)$ acts injectively in semi-fibered concordance, where $P$ is a pattern of non-zero winding number and $J \sqcup K$ is in SSF. Virtual covers can also be applied to knots in closed fibered $3$-manifolds. Applications are sketched in Section \ref{sec_three}.  Section \ref{sec_back} contains a review of virtual covers, virtual knot concordance, and concordance invariants of virtual knots.

\section{Background}\label{sec_back}

\subsection{Links and Virtual Knots} We will assume the reader has some familiarity with virtual knot theory. Recall that there are four models for virtual knots: (1) virtual knot diagrams modulo extended Reidemeister moves \cite{KaV}, (2) Gauss diagrams modulo diagrammatic versions of the Reidemeister moves \cite{GPV}, (3) abstract link diagrams modulo Kamada-Kamada equivalence \cite{kamkam}, (4) knots in thickened surfaces modulo stabilization/destabilization and stable diffeomorphism \cite{kuperberg}. Given a representative $R$ from any of these four models of virtual knots, we will denote by $\kappa(R)$ the equivalence class of virtual knot diagrams from model (1). Equivalence of (ordered, oriented) links in $\mathbb{S}^3$ and oriented virtual knots is denoted by ``$\leftrightharpoons$''.

\subsection{Concordance of Knots and Virtual Knots} We work throughout in the smooth category. Two oriented knots $K_0$ and $K_1$ in $\mathbb{S}^3$ are said to be concordant in $\mathbb{S}^3$ if there is an embedded oriented annulus $A$ in $\mathbb{S}^3 \times \mathbb{I}$ such that $A\cap \mathbb{S}^3 \times \{i\}=(-1)^i K_i$, where $-K$ denotes a change of orientation. If $K_0$ and $K_1$ are concordant, we write $K_0 \asymp K_1$. A concordance can be realized combinatorially as a sequence of births $b$ (local minimums), deaths $d$ (local maximums), and saddle moves $s$ such that $\# b-\#s+\#d=0$.

\begin{figure}[htb]
\begin{tabular}{|c|c|}  \hline
\begin{tabular}{c} \\
$K \sqcup \begin{array}{c} \def\svgwidth{.35in}
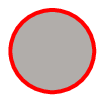 \end{array} $\\ \\ \ $\text{birth} \uparrow \,\,\,\, \downarrow \text{death}$ \\ \\
 $K$
\end{tabular}
& 
\begin{tabular}{ccc}
\multicolumn{3}{c}{\underline{saddle move}} \\ & &\\
\begin{tabular}{c}
\def\svgwidth{1in}
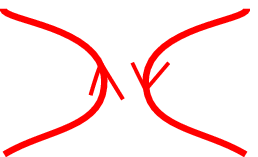 
\end{tabular}
& $\leftrightarrow$ &
\begin{tabular}{c}
\def\svgwidth{1in}
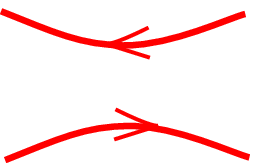 
\end{tabular}  \end{tabular} \\ \hline
\end{tabular} 
\caption{Births, Deaths, and Saddles.}\label{fig_concordance}
\end{figure}

Let $\Sigma$ be a compact connected oriented (c.c.o.) surface.  For knots in $\Sigma \times \mathbb{R}$, we have notion of concordance due to Turaev \cite{turaev_cobordism}. We denote a knot $K$ in a $3$-manifold $N$ by $K^N$. For $i=0,1$, let $\mathfrak{k}_i$ be knots in $\Sigma_i \times \mathbb{R}$. Then $\mathfrak{k}_0^{\Sigma_0 \times \mathbb{R}}, \mathfrak{k}_1^{\Sigma_1 \times \mathbb{R}}$ are \emph{concordant} if there is a c.c.o. $3$-manifold $M$, an embedding of the surface $\Sigma_0 \sqcup -\Sigma_1 \to \partial M$, and a properly embedded oriented annulus $\mathfrak{a}:\mathbb{S}^1 \times \mathbb{I} \to M \times \mathbb{R}$ such that for $i=0,1$, $ \mathfrak{a}\cap (\Sigma_i \times \mathbb{R})=(-1)^i\mathfrak{k}_i^{\Sigma_i \times \mathbb{R}}$.  Again we denote concordant knots in thickened surfaces by $\mathfrak{k}_0^{\Sigma_0 \times \mathbb{R}} \asymp \mathfrak{k}_1^{\Sigma_1 \times \mathbb{R}}$. 
\newline
\newline
Two oriented virtual knots $\upsilon_0$, $\upsilon_1$ are \emph{concordant} if they are obtained from one another by a finite sequence of extended Reidemeister moves, births $b$, deaths $d$, and saddle moves $s$ satisfying $\#b-\#s+\#d=0$. This combinatorial definition, introduced in \cite{DKK}, was shown to be equivalent to a geometric formulation in \cite{CKS}: the concordance relation for virtual knots is equivalent to concordance relation of knots in thickened surfaces together with stabilization/destabilization (see \cite{CKS}, Lemma 4.6).

\subsection{Concordance Invariants of Virtual Knots} We will use two invariants to separate concordance classes of virtual knots: the Henrich-Turaev (HT) polynomial and the slice genus. Let $\upsilon$ be an oriented virtual knot diagram. Let $G$ be a Gauss diagram for $\upsilon$ and $x$ an arrow of $G$. Let $\bowtie(G)$ be the set of arrows of $G$. Suppose that $a\in\bowtie(G)\backslash\{x\}$ intersects $x$. Define $\text{int}_x(a)=\pm 1$ according to Figure \ref{fig_int_defn}. Then define:
\[
\text{index}(x)=\sum_a \text{sign}(a) \text{int}_x(a), \,\,\,
w_{\upsilon}(t)=\sum_{\stackrel{x \in \bowtie (G)}{\text{index}(x) \ne 0}} \text{sign}(x) t^{|\text{index}(x)|},
\]
where $\text{sign}(a)=\pm 1$ denotes the local writhe of the crossing of $a$. The polynomial $w_{\upsilon}(t)$ is called the \emph{HT}  \footnote{This definition of the HT polynomial is slightly different from the definition in \cite{henrich, turaev_cobordism}}  \emph{polynomial} \cite{henrich, turaev_cobordism}.

\begin{figure}[htb]
\begin{tabular}{|cc|c|} \hline & &  \\
\begin{tabular}{c}
\def\svgwidth{.75in}
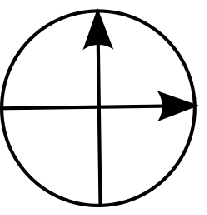 \end{tabular} & \begin{tabular}{c} \def\svgwidth{.8in}
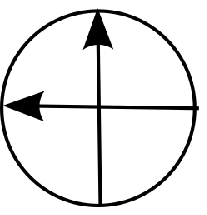 \end{tabular} & \begin{tabular}{c}\def\svgwidth{1.5in}
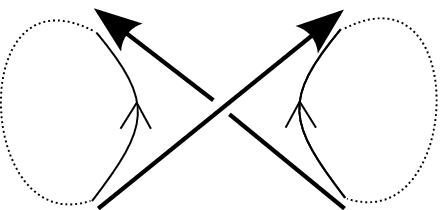 \end{tabular} \\ \hline
\end{tabular}
\caption{(Left) Definition of $\text{int}_x(a)$ and (Right) distinguished halves.} \label{fig_int_defn}
\end{figure}

\begin{theorem} \label{thm_ht_invar} The HT polynomial is a concordance invariant of virtual knots.
\end{theorem}
\begin{proof} Suppose that $\mathfrak{k}_0^{\Sigma_0 \times \mathbb{R}} \asymp \mathfrak{k}_1^{\Sigma_1 \times \mathbb{R}}$, where each $\mathfrak{k}_i$ has a regular projection to $\Sigma_i$. Then we may determine $\kappa(\mathfrak{k}_i)$ by finding a Gauss diagram $G_i$ for $\mathfrak{k}_i$ on $\Sigma_i$ in the usual way. If $x \in \bowtie (G_i)$, $|\text{index}(x)|$ may be computed by performing the oriented smoothing at $x$ (see Figure \ref{fig_int_defn}) and taking the algebraic intersection number of the resulting two curves $D_x$, $D_x'$ on $\Sigma$ (for details, see \cite{mwc_had}). If $D$ is the immersed curve of the diagram of $\mathfrak{k}_i$ in $\Sigma_i$, then $[D]=[D_x]+[D_x']$ in $H_1(\Sigma;\mathbb{Z})$. Hence $|D\cdot D_x|=|D_x \cdot D_x'|$. This implies that:
\[
w_{\kappa(\mathfrak{k}_i)}(t)=\sum_{x \in \bowtie (G_i), D \cdot D_x \ne 0} \text{sign}(x) t^{|D_x \cdot D|}.
\]
In \cite{turaev_cobordism}, Turaev defines two concordance invariants  $u_+,u_-$ of knots in thickened surfaces. The above formula shows that $w_{\kappa(\mathfrak{k}_i)}(t)=u_+(\mathfrak{k}_i)(t)+u_-(\mathfrak{k}_i)(t)$. Hence, $w_{\kappa(\mathfrak{k}_0)}(t)=w_{\kappa(\mathfrak{k}_1)}(t)$. 
\end{proof}

If the condition that $\#b-\#s+\#d=0$ is removed from virtual knot concordance, then we have \emph{virtual knot cobordism}. The \emph{slice genus} of a virtual knot $\upsilon$ is the smallest genus of all formal cobordisms taking $\upsilon$ to the unknot. For a virtual knot with all positive crossings, i.e. a positive virtual knot, the slice genus can be computed using a generalization of Rasmussen's theorem \cite{DKK}. Recall that if the oriented smoothing is preformed at all classical crossings of $\upsilon$, then the resulting set of immersed curves is in one-to-one correspondence with the \emph{virtual Seifert circles} of $\upsilon$. The virtual Seifert circles of a classical knot are the Seifert circles. 

\begin{theorem}[Generalization of Rasmussen \cite{DKK}] \label{thm_gen_ras} Let $\upsilon$ be a virtual knot diagram whose classical crossings are all positively signed. Let $r$ be the number of virtual Seifert circles of $\upsilon$ and $n$ the number of classical crossings of $\upsilon$. Then the slice genus is $(-r+n+1)/2$.
\end{theorem}

\subsection{Virtual Covers of Links}\label{sec_virt_cov} Here we review virtual covers \cite{vc_1}. Let $N$ be a c.c.o. $3$-manifold admitting a regular orientation preserving (o.p.) covering projection $\Pi:\Sigma \times \mathbb{R} \to N$, where $\Sigma$ is a c.c.o. $2$-manifold.  A \emph{lift by} $\Pi$ of $K^N$ is an oriented knot $\mathfrak{k}^{\Sigma \times \mathbb{R}}$ such that $\Pi(\mathfrak{k}^{\Sigma \times \mathbb{R}})=K^N$, with orientations preserved. The triple $(\mathfrak{k}^{\Sigma \times \mathbb{R}},\Pi,K^N)$ specifying a lift is called a \emph{virtual cover}. The virtual knot $\kappa(\mathfrak{k}^{\Sigma \times \mathbb{R}})$ is called the \emph{associated virtual knot}. Let $\Upsilon(K^N)$ denote the set of all associated virtual knots for all lifts by $\Pi$ of $K^N$. If $|\Upsilon(K^N)|=1$, then we call the unique element the \emph{invariant associated virtual knot}. The following lemma from \cite{vc_1}, is essentially the main result in \cite{cm_fiber} in a more compact form.

\begin{lemma} \label{lemma_cm_general} Suppose there are virtual covers $(\mathfrak{k}_0^{\Sigma \times \mathbb{R}},\Pi,K_0^N),(\mathfrak{k}_1^{\Sigma \times \mathbb{R}},\Pi,K_1^N)$ with invariant associated virtual knots $\upsilon_0,\upsilon_1$, respectively. If $K_0^N \leftrightharpoons K_1^N$, then $\upsilon_0 \leftrightharpoons \upsilon_1$.
\end{lemma}

The canonical example of a virtual cover is a knot $K$ in the complement of a fibered knot $J$ where $\text{lk}(J,K)=0$. Set $N_J=\overline{\mathbb{S}^3 \backslash V(J)}$, where $V(*)$ denotes a tubular neighborhood of $*$. Since $J$ is fibered, it has a Seifert surface $\Sigma_J$ such that the pair $(\overline{\mathbb{S}^3\backslash V(\Sigma_J)}, \overline{\mathbb{S}^3\backslash V(\Sigma_J)}  \cap \partial N_J)$ is diffeomorphic as a pair to $(\Sigma_J \cap N_J,\partial (\Sigma_J \cap N_J)) \times \mathbb{I}$ \cite{kawauchi}. Then $N_J$ may be identified with a mapping torus $(\Sigma_J \cap N_J) \times \mathbb{I}/\psi$ where $\psi: (\Sigma_J \cap N_J) \to (\Sigma_J \cap N_J)$ is an o.p. diffeomorphism. This gives a regular covering projection $\Pi_J: (\Sigma_J \cap N_J) \times \mathbb{R} \to N_J$. 
\newline
\newline
Since $\text{lk}(J,K)=0$, there is a virtual cover $(\mathfrak{k}^{\Sigma_J \times \mathbb{R}},\Pi_J,K^{N_J})$. The associated virtual knot is invariant when $K$ is in \emph{special Seifert form} (SSF) with respect to $\Sigma_J$.  Briefly, we say $K$ is in SSF with respect to $\Sigma_J$ if its image can be decomposed as a disjoint union of a finite number of embedded intervals on $\Sigma_J$ and a finite number of ``crossings''. Each ``crossing'' consists of a pair of disjoint arcs lying in different hemispheres of an embedded $3$-ball $B$ satisfying $B \cap \Sigma_J \approx \mathbb{B}^2$. A formal definition is given in \cite{cm_fiber}. See also \cite{vc_1} for further discussion.

\begin{figure}[htb]
\begin{tabular}{|ccc|} \hline & & \\
\begin{tabular}{c}
\def\svgwidth{1.6in}
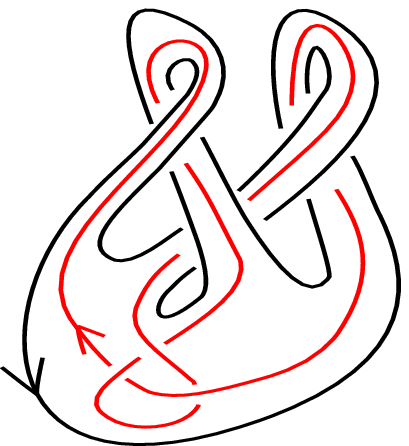 \end{tabular} & $\to$ & \begin{tabular}{c} \def\svgwidth{1.4in}
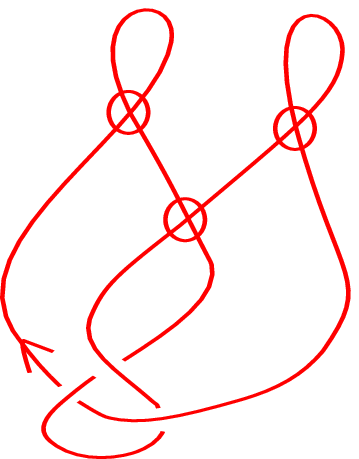 \end{tabular} \\ \hline
\end{tabular}
\caption{A link $J \sqcup K$ in SSF (left) and its invariant associated virtual $\upsilon$ (right).} \label{fig_ssf}
\end{figure}

If $K$ is in SSF with respect to a fiber $\Sigma_J$, let $[K;\Sigma_J]$ denote the oriented knot diagram on $\Sigma_J$ made of the arcs and crossings of the SSF. In \cite{cm_fiber} it was shown that $\upsilon\leftrightharpoons\kappa([K;\Sigma_J])$ is the invariant associated virtual knot. See Figure \ref{fig_ssf} for a typical example. The main result of \cite{vc_1} showed that $\upsilon$ functions essentially as an invariant of links in SSF:

\begin{theorem} \label{theorem_link_invar} For $i=0,1$, let $L_i=J_i \sqcup K_i$ be oriented ordered links with $J_i$ a fibered link, $K_i$ in SSF with respect to a fiber $\Sigma_i$ of $J_i$. For $i=0,1$, let $\upsilon_i$ be the invariant associated virtual knot for $L_i$. If $L_0 \leftrightharpoons L_1$, then $\upsilon_0 \leftrightharpoons \upsilon_1$.
\end{theorem}

\section{Semi-Fibered Concordance}

This section gives the precise definitions of fibered concordance and semi-fibered concordance. It is shown that the concordance class of the invariant associated virtual knot is a semi-fibered concordance invariant of two component links in SSF. Furthermore, it is shown this concordance class can be used to separate an infinite number of semi-fibered concordance classes of two component links.

\subsection{Semi-Fibered Concordance of Links} \label{sec_semi_fib} Let $J_0,J_1$ be oriented fibered knots in $\mathbb{S}^3$ with fiber bundle projections $p_0:\overline{\mathbb{S}^3 \backslash V(J_0)} \to \mathbb{S}^1$, $p_1:\overline{\mathbb{S}^3 \backslash V(J_1)} \to \mathbb{S}^1$, resp. Suppose there is an annulus $A$ embedded in $\mathbb{S}^3 \times \mathbb{I}$ such that $A \cap \mathbb{S}^3 \times \{i\}=(-1)^iJ_i$ for $i=0,1$ and such that there is a fiber bundle projection $p:\overline{\mathbb{S}^3\times \mathbb{I} \backslash V(A)} \to \mathbb{S}^1$ where $p|_{\mathbb{S}^3 \times \{i\}}=p_i$ for $i=0,1$. Then $J_0$ and $J_1$ are said to be \emph{fibered concordant}, denoted $J_0 \asymp_f J_1$ \cite{harer}.
\newline
\newline
When $J_0 \asymp_f J_1$, there is a covering projection $\Pi:H \times \mathbb{R} \to \overline{\mathbb{S}^3\times \mathbb{I} \backslash V(A)}$, where the fiber $H$ is a c.c.o. 3-manifold. Moreover, there are oriented Seifert surfaces $\Sigma_0$ of $J_0$ and $\Sigma_1$ of $J_1$ such that $\Sigma_0 \sqcup -\Sigma_1 \hookrightarrow \partial H$ and $\Pi|_{\Sigma_0 \times \mathbb{R}}, \Pi|_{\Sigma_1 \times \mathbb{R}}$ are the infinite cyclic covers of $J_0,J_1$, respectively.

\begin{remark} It is known that there are knots that are not concordant to a fibered knot \cite{livingston_survey}. Every fibered knot in $\mathbb{S}^3$ is fiber concordant to a hyperbolic fibered knot \cite{soma}.
\end{remark}
Let $L_0=J_0 \sqcup K_0$, $L_1=J_1 \sqcup K_1$ be two component links with $J_i$ fibered and $J_0 \asymp_f J_1$ via an annulus $A_J$ in $\mathbb{S}^3 \times \mathbb{I}$. If $K_0 \asymp K_1$ via an annulus $A_K$ in $\mathbb{S}^3 \times \mathbb{I}$ with $A_K \cap A_J =\emptyset$, then we will say that $L_0$ and $L_1$ are \emph{semi-fibered concordant}. This will be denoted by $L_0 \asymp_{sf} L_1$. 

\begin{lemma} \label{lemm_sf_invar}Let $L=J \sqcup K$ with $J$ fibered and $K$ in SSF with respect to some fiber $\Sigma_J$ and $\upsilon$ the invariant associated virtual knot. If $L \leftrightharpoons J_1 \sqcup K_1$, then $L \asymp_{sf} J_1 \sqcup K_1$, $K_1$ is in SSF with respect to some fiber $\Sigma_{J_1}$, and $\upsilon \leftrightharpoons \kappa([K_1;\Sigma_{J_1}])$.
\end{lemma}
\begin{proof}  Let $F:\mathbb{S}^3 \times \mathbb{I}\to \mathbb{S}^3$ be the ambient isotopy taking $L$ to $J_1 \sqcup K_1$. Let $A_J$ be the annulus defined as the image of the map $a_J:\mathbb{S}^1 \times \mathbb{I} \to \mathbb{S}^3 \times \mathbb{I}$, $a_J(z,t)=(F(J(z),t),t)$. Similarly define an embedded annulus $A_K$. Then $J \asymp_f J_1$ via $A_J$ and $A_K \cap A_J=\emptyset$. Thus $L \asymp_{sf} J_1 \sqcup K_1$. Set $\Sigma_{J_1}=F(\Sigma_J,1)$. Then $K_1=F(K_0,1)$ is in SSF with respect to $\Sigma_{J_1}$ (see \cite{vc_1}, Lemma 2). The last claim follows from Theorem \ref{theorem_link_invar}.
\end{proof}

\begin{theorem}[Main Theorem]\label{theorem_sf_invar} Let $L_0=J_0 \sqcup K_0$, $L_1=J_1 \sqcup K_1$ be two component links with $J_0,J_1$ fibered. Suppose for $i=0,1$, $K_i$ in SSF with respect to some fiber $\Gamma_i$ of $J_i$ and let $\upsilon_i$ be the invariant associated virtual knot $\upsilon_i$. If $L_0 \asymp_{sf} L_1$, then $\upsilon_0 \asymp \upsilon_1$.
\end{theorem}

\begin{proof} Let $A_J,A_K,N,\Pi, H, \Sigma_0,\Sigma_1$ be as in the definitions above. The fibers $\Gamma_i$ and $\Sigma_i$ are minimal genus Seifert surfaces. Minimal genus Seifert surfaces of fibered knots $J_i$ are unique up to ambient isotopy in $\mathbb{S}^3$ acting as the identity on $J_i$ \cite{whitten}. Thus for $i=0,1$, we may use Lemma \ref{lemm_sf_invar} to move $\Gamma_i$ and $K_i$ so that we have a new knot in SSF with respect to $\Sigma_i$. As this does not affect the semi-fibered concordance classes of $L_i$ or the concordance class of $\upsilon_i$, we may as well assume from the beginning that $\Gamma_i=\Sigma_i$.
\newline
\newline
Now, there are virtual covers $(\mathfrak{k}_i^{\Sigma_i \times \mathbb{R}},\Pi|_{\Sigma_i \times \mathbb{R}}, K_i^{N_{J_i}})$, $i=0,1$. By the homotopy lifting theorem, $A_K$ lifts to a smoothly embedded annulus $\mathfrak{a}:\mathbb{S}^1 \times \mathbb{I} \to H \times \mathbb{R}$ such that $\mathfrak{a}(\mathbb{S}^1,0)$ is identified with $\mathfrak{k}_0^{\Sigma_0 \times \mathbb{R}}$. Set $\mathfrak{l}=\mathfrak{a}(\mathbb{S}^1,1)$ so that we have another virtual cover $(\mathfrak{l}^{\Sigma_1 \times \mathbb{R}}, \Pi|_{\Sigma_1 \times \mathbb{R}}, K_1^{N_{J_1}})$. Orient $\mathfrak{l}$ appropriately, so that it matches the orientation of lifts by $\Pi|_{\Sigma_1 \times \mathbb{R}}$ of $K_1^{N_{J_1}}$. 
\newline
\newline
Then $\mathfrak{k}_0$ and $\mathfrak{l}$ are concordant as knots in thickened surfaces. Hence we have that $\kappa(\mathfrak{k}_0)\asymp\kappa(\mathfrak{l})$ as virtual knots (by \cite{CKS}, Lemma 4.6). Since equivalent virtual knots are concordant, and all lifts of $K_1$ by $\Pi|_{\Sigma_1\times \mathbb{R}}$ stabilize to the same virtual knot (by hypothesis), $\upsilon_1 \asymp \kappa(\mathfrak{l})$ and $\upsilon_0 \asymp \kappa (\mathfrak{k}_0)$. This completes the proof.
\end{proof}

The following theorem shows that the associated virtual knots distinguish a large set of semi-fibered concordance classes. Moreover, it shows that the HT polynomial is useful at separating these classes. It is inspired by \cite{UK}, Theorem 1, and \cite{turaev_cobordism}, Theorem 1.6.1.

\begin{figure}[htb]
\begin{tabular}{|c|c|} \hline & \\
\begin{tabular}{c}
\def\svgwidth{4in} 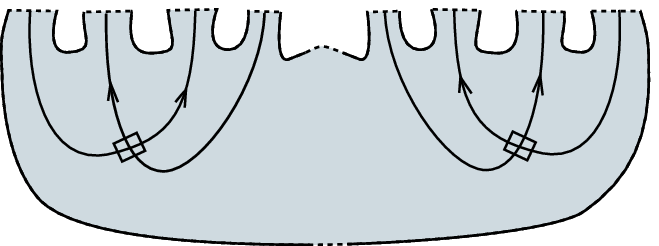 \end{tabular} &
\begin{tabular}{cc}
\def\svgwidth{.75in} 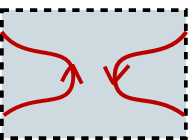 & 
\def\svgwidth{.75in} 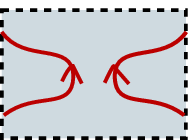 \\
$\downarrow $ & $\downarrow$ \\
\def\svgwidth{.75in} 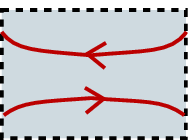 &
\def\svgwidth{.75in} 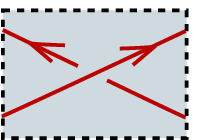 \\
\multicolumn{2}{c}{\underline{join by surgeries}}
\end{tabular} 
 \\ \hline
\end{tabular}
\caption{A fiber $\Sigma_J$ in disc-band form.} \label{fig_disc_band_inf}
\end{figure}

\begin{theorem} \label{thm_ht_nontriv} Let $J$ be a fibered knot and $\Sigma_J$ a fiber of genus $g$ in disc-band form, with symplectic basis $a_1,b_1,\ldots,a_g,b_g$ . Let $h=\sum_{i=1}^g p_i a_i+q_i b_i \in H_1(\Sigma_J;\mathbb{Z})$ such that for all $i$, $p_i,q_i \ne 0$, and $\gcd(|p_i|,|q_i|)=1$. Then there is an infinite set of pairwise non-semi-fibered  concordant non-split links $J \sqcup K_k$ with $K_k$ in SSF with respect to $\Sigma_J$, $[K_k;\Sigma_J]$ represents $h$, and $w_{\upsilon_k}(t) \ne 0$, where $\upsilon_k$ is the invariant associated virtual knot.
\end{theorem}
\begin{proof} Since $\Sigma_J$ is in disc-band form \cite{bz}, we may assume that the $\Sigma_J$ and the symplectic  basis is as depicted in Figure \ref{fig_disc_band_inf}. Curves representing $a_i,b_i$ are also denoted as $a_i,b_i$. For each $i$, draw $p_i$ (resp. $q_i$) simple closed curves on $\Sigma_J$ parallel to $a_i$ (resp. $b_i$). This makes $|p_i q_i|$ intersections in a neighborhood of the point $a_i \cap b_i$ (see Figure \ref{fig_ht_nontriv_inf}, left). For each intersection, substitute the same picture in a purple box from Figure \ref{fig_ht_nontriv_inf}, center: top left if $p_i,q_i>0$, top right if $p_i,q_i<0$, bottom left if $p_i>0$ and $q_i<0$, bottom right if $p_i<0$ and $q_i>0$. 
\newline
\newline
Since $\gcd(|p_k|,|q_k|)=1$, we have $g$ knot diagrams $[M_1;\Sigma_J],\ldots,[M_g;\Sigma_J]$, oriented by choice of a purple box. For $k\ge 1$, insert the long $(2k+1,2)$ torus knot in a small arc on $[M_1; \Sigma_J]$ away from any of the purple boxes (as in the orange circle in Figure \ref{fig_ht_nontriv_inf}). These knot diagrams may be surgered together from left to right in $\Sigma_J$ as in Figure \ref{fig_disc_band_inf}, right. Call the resulting knot diagram $[K_k;\Sigma_J]$. Clearly, this corresponds to an oriented knot $K_k$ in SSF with respect to $\Sigma_J$. Now compute $w_{\upsilon_k}(t)=2 t \cdot \sum_{i=1}^g |p_iq_i| \ne 0$. Hence, $J \sqcup K_k$ is non-split \cite{cm_fiber}. Lastly, Theorem \ref{thm_gen_ras} implies that increasing $k$ by $1$ increases the slice genus of $\upsilon_k$ by $1$.
\end{proof}

\begin{figure}[htb]
\begin{tabular}{|ccc|} \hline & & \\
\begin{tabular}{c}
\def\svgwidth{1.9in}
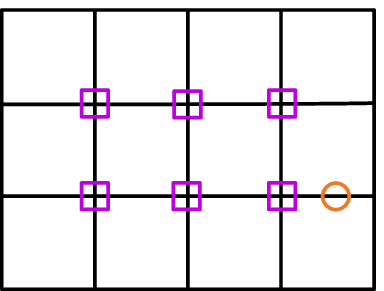 \end{tabular} & 
\begin{tabular}{cc}
\def\svgwidth{.75in} 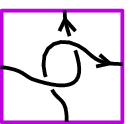 & 
\def\svgwidth{.75in} 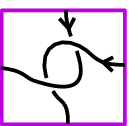 \\
\def\svgwidth{.75in} 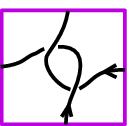 &
\def\svgwidth{.75in} 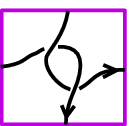 \\
\end{tabular} 
&
\begin{tabular}{c}
\def\svgwidth{1.3in} 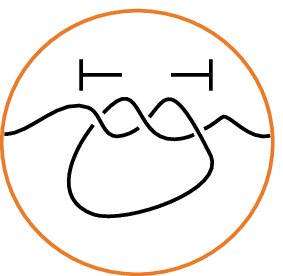
\end{tabular}
\\ \hline
\end{tabular}
\caption{The construction used in the proof of Theorem \ref{thm_ht_nontriv}.} \label{fig_ht_nontriv_inf}
\end{figure}
\section{Ribbon and Slice Obstructions} \label{sec_sli_rib} Now we will use the results of the previous section to identify slice and ribbon obstructions. Recall that a knot in $\mathbb{S}^3$ is said to be \emph{slice} if $K \asymp \bigcirc$, with $\bigcirc$ bounding an embedded disc. A knot is said to be \emph{ribbon} in $\mathbb{S}^3$ if it bounds an immersed disc $\mathbb{B}^2$ in $\mathbb{S}^3$ having only ribbon singularities: the pre-image of any singular arc is two disjoint simple arcs $a_b$ and $a_i$, where $a_b$ intersects $\partial \mathbb{B}^2$ in its endpoints and $a_i \subseteq \text{int}(\mathbb{B}^2)$. We will use the same definition for slice and ribbon in any oriented $3$-manifold $N$: just replace $\mathbb{S}^3$ with $N$.
\newline
\newline
A ribbon knot $K^N$ is a band connected sum in $N$ of the boundaries of $n$ disjoint discs embedded in $N$ (i.e. an $n$ component unlink in $N$). Performing saddle moves on the bands gives $n$ disjoint $2$-discs which may be eliminated with deaths. Likewise, a virtual knot $\upsilon$ is said to be \emph{ribbon} if there is concordance of $\upsilon$ with the unknot consisting of a sequence of extended Reidemeister moves, deaths, and saddle moves. Thus, if $\mathfrak{k}^{\Sigma \times \mathbb{R}}$ is a ribbon knot in the thickened c.c.o. surface $\Sigma \times \mathbb{R}$, then its stabilization $\kappa(\mathfrak{k}^{\Sigma \times \mathbb{R}})$ is a ribbon virtual knot.   

\begin{theorem} \label{thm_ribbon}(with R. Todd) Let $L=J \sqcup K$ with $J$ fibered and $K$ in SSF with respect to some fiber $\Sigma$ of $J$ and $\upsilon$ the invariant associated virtual knot. 
\begin{enumerate}
\item If $L \asymp_{sf} J_0 \sqcup \bigcirc$, with $\bigcirc$ bounding an embedded disc in the complement of $J_0$, then $\upsilon\asymp \bigcirc$ (i.e. it is a slice virtual knot).
\item If $K$ bounds a ribbon disc disjoint from $J$, then $L \asymp_{sf} J \sqcup \bigcirc$, where $\bigcirc$ bounds a disc in the complement of $J$ and $\upsilon$ is a ribbon virtual knot.
\end{enumerate}
\end{theorem}
\begin{proof} The first claim is an immediate consequence of Theorem \ref{theorem_sf_invar}. For the second claim, let $A_J$ be the annulus $J \times \mathbb{I}$ in $\mathbb{S}^3 \times \mathbb{I}$. Then $J \asymp_f J$ via $A_J$ and we have the covering projection $\Pi:(\Sigma \times \mathbb{I}) \times \mathbb{R} \to \overline{\mathbb{S}^3 \times \mathbb{I}\backslash V(A_J)}$. Let $D:\mathbb{B}^2 \to \overline{\mathbb{S}^3\backslash V(J)}$ denote the immersion of the ribbon disc for $K$ and by abuse of notation, the ribbon disc itself. Using saddle moves and deaths on $D$, we see that $L \asymp_{sf} J \sqcup \bigcirc$, where $\bigcirc$ bounds a disc in $\mathbb{S}^3$ disjoint from $J$.  
\newline
\newline
Now let $(\mathfrak{k}^{\Sigma \times \mathbb{R}},\Pi_J,K^{N_J})$ be a virtual cover. To see that $\upsilon$ is ribbon, note that $D$ lifts to an immersed disc $\mathfrak{d}:\mathbb{B}^2 \to \Sigma \times \mathbb{R}$  with $\partial \mathfrak{d}=\mathfrak{k}$ (essentially by the chain rule and lifting criterion). Consider a ribbon singularity of $D$ as a path in $D$ connecting two points of $K$. This lifts to a path in $\mathfrak{d}$ connecting two points of $\mathfrak{k}$. Since $\Pi_J(\mathfrak{d})=D$, these are ribbon singularities of $\mathfrak{d}$. Moreover, $\mathfrak{d}$ can have no other singularities. Using saddle moves and deaths on $\mathfrak{d}$ in $(\Sigma \times \mathbb{I}) \times \mathbb{R}$, we see that $\mathfrak{k}$ is ribbon. Hence $\upsilon$ is a ribbon virtual knot.
\end{proof}

\textbf{Example:} Recall that $J \sqcup K$ is a \emph{ribbon link} if it is the boundary of an immersed $\mathbb{B}^2 \sqcup \mathbb{B}^2$ all of whose singularities are ribbon singularities.  The left hand side of Figure \ref{fig_sato_levine_1} shows a two component link $L=J \sqcup K$ both of whose components are ribbon. Indeed, $J$ is a square knot and $K$ is a trivial knot. The square knot is a fibered knot with a fiber $\Sigma_J$ drawn as in the figure. Note that $K$ is in SSF relative to $\Sigma_J$. An invariant associated virtual knot $\upsilon$ can be found. Note that $w_{\upsilon}(t)=-2t \ne 0$. Hence, $K$ cannot bound a ribbon disc in the complement of $J$. This also implies that $L$ is not a \emph{pure ribbon link} i.e. a ribbon link where the components bound disjoint ribbon discs.
\newline
\newline
A \emph{boundary link} is a link where the components bound disjoint Seifert surfaces. Any pure ribbon link is a boundary link because the ribbon singularities may be modified as in \cite{cochran_deriv} to obtain Seifert surfaces. Cochran's generalized Sato-Levine invariant  $\beta(L)$ \cite{cochran_geom} is a $\mathbb{Z}^{\infty}$-valued link concordance invariant that is equal to $(0,0,0,\ldots)$ on boundary links, hence on pure ribbon links.  To compute the invariant begin with a link $J \sqcup K$ with $\text{lk}(J,K)=0$ and find a Seifert surface $\Sigma_K$ of $K$ that does not intersect $J$. $\Sigma_K$ may be chosen so that $\Sigma_K \cap \Sigma_J$ is connected; call this component a \emph{derivative} $D(L)$. It is oriented following the convention of \cite{cochran_geom}. The first coordinate of the Sato-Levine invariant is $\text{lk}(D(L)^+,D(L))$, where $D(L)^+$ is the positive push off of $D(L)$ from $\Sigma_J$. This process is iterated to find the remaining coordinates.
\newline
\newline
For our link $L=J \sqcup K$, a convenient $D(L)$ is given on the right of Figure \ref{fig_sato_levine_1}. The first coordinate of $\beta(L)$ is zero. It is easy to see that $J \sqcup D(L)$ is a ribbon link (e.g. it is a symmetric union \cite{eisermann_lamm}). All the remaining coordinates in $\beta(L)$ are zero. Thus the associated virtual knot detects that $L$ is not a pure ribbon link, whereas $\beta(L)$ does not. Note that the Jones polynomial nullity detects this link as non-ribbon  \cite{eisermann_ribbon}.  However the associated virtual knot detects something subtler: every ribbon disc for $K$ must intersect $J$. \hfill $\square$
\newline
\begin{figure}[htb]
\begin{tabular}{|ccc|} \hline & & \\
\def\svgwidth{2in}
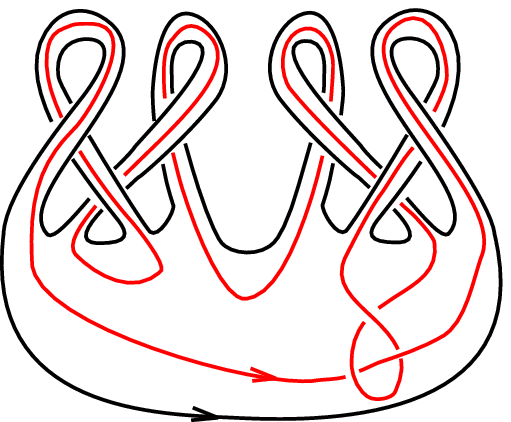 & & \def\svgwidth{2in}
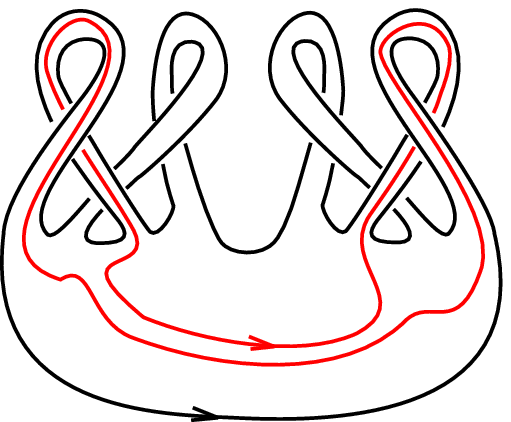\\ \hline
\end{tabular}
\caption{A link that is not pure ribbon (left) and a derivative  used for computation of the Sato-Levine Invariant (right).} \label{fig_sato_levine_1}
\end{figure}

\textbf{Example:} We return now to our motivating examples from \cite{gst}. Recall they are an infinite family of two component slice links $L_n=J_0 \sqcup V_n$ with $J_0$ a square knot and $V_n$ a smooth simple closed curve on a fiber $\Sigma_J$ of $J$ (by \cite{scharlemann}). Thus, $V_n$ is in SSF with respect to $\Sigma_J$. The SSF for $V_n$ has no crossings in balls, so the invariant associated virtual knot $\upsilon_n \leftrightharpoons \bigcirc$. Since $\upsilon_n$ is a ribbon virtual knot, one might hope to construct a ribbon presentation for $L_n$ using virtual covers. Such a construction would require a fibered concordance of $J_0$ with the unknot obtained by pushing an immersed ribbon disc for $J_0$ into $\mathbb{B}^4$.  A method for doing exactly this was discovered by Aitchison and Silver \cite{silver_fibred}.
\newline
\newline 
Here we describe the method in brief. A slice disc $D'$ of $J_0$ in $\mathbb{B}^4$ together with the fiber $\Sigma_J$ bound a solid two holed torus $H$.  The Alexander polynomial of $J_0$ specifies an automorphism $\phi$ of $\pi_1(H,*)\cong \left<x_1,x_2|-\right>$ (called the \emph{monodromy}). An o.p. diffeomorphism $\psi:H \to H$ is constructed from $\phi$. On the left in Figure \ref{fig_square_knot_disc}, $H$ is obtained by identifying each of the ellipses on top to the ellipse immediately below it. The yellow bands indicate the image of $x_1$ and $x_2$ under $\phi$, as in \cite{silver_fibred}; the thick dumbbell represents possible positions for $*$. Lastly, a handlebody decomposition of $\overline{\mathbb{B}^4\backslash V(D')}$ is recovered from the mapping torus defined by $\psi$. The ribbon disc $D$ on the right in Figure \ref{fig_square_knot_disc} is also obtainable from the handelbody decomposition.
\newline
\newline 
\underline{Problem:} Either construct a ribbon presentation for each $L_n$ such that the ribbon disc for $J_0$ is $D$ (see Figure \ref{fig_square_knot_disc}, right), or show that no such presentation exists. \hfill $\square$

\begin{figure}
\begin{tabular}{|c|} \hline \\
\xymatrix{
\begin{tabular}{c}
\def\svgwidth{1.9in}
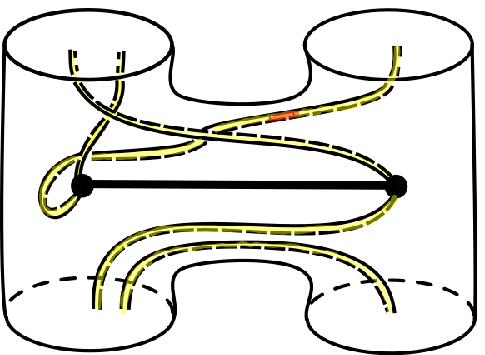\end{tabular} \ar[r]&  
\begin{tabular}{c}
\def\svgwidth{1.75in}
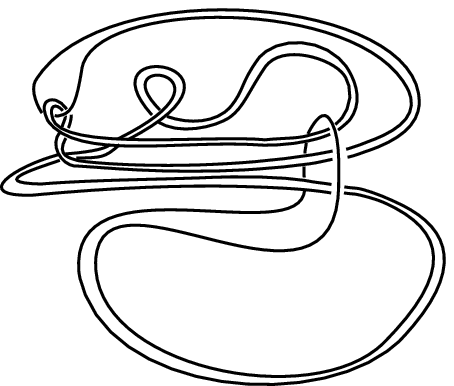 \end{tabular}} 
\\
\hline
\end{tabular}
\caption{Construction of a ribbon disc for the square knot using the method of Aitchison and Silver.}\label{fig_square_knot_disc}
\end{figure}

\section{Injectivity of Satellite Operators}
\label{sec_sat}
Take a knot $P$ in a solid torus and tie it into the shape of a knot $K$ in $\mathbb{S}^3$. Roughly speaking, this is the satellite knot $P(K)$ with pattern $P$ and companion $K$. For a fixed pattern $P$, do non-concordant companions yield non-concordant satellites with pattern $P$, i.e. is the map $K \to P(K)$ injective in concordance? Here we consider this question for the map $J \sqcup K \to J \sqcup P(K)$ in semi-fibered concordance. We begin with some preparatory results.

\subsection{Satellite Operators and the HT polynomial} Let $K$ be a knot in $\mathbb{S}^3$ and $l$ be a longitude on a closed tubular neighborhood $\overline{V}(K)$ with $\text{lk}(K,l)=0$. Let $P$ be a knot in $\overline{V}=\mathbb{S}^1 \times \mathbb{B}^2$ that is not contained in a $3$-ball in $B$. Let $f:\overline{V} \to \overline{V}(K)$ be an o.p. diffeomorphism such that $f(\mathbb{S}^1 \times \{0\})$ is identified with $K$, $f$ takes a meridian of $\overline{V}$ to a meridian of $\overline{V}(K)$, and the longitude $\mathbb{S}^1 \times \{1\}$ is mapped to $l$. Then the image $P(K)$ of $P$ under $f$ in $\mathbb{S}^3$ is an oriented knot called the \emph{untwisted satellite} with \emph{pattern} $P$ and \emph{companion} $K$. The knot $\mathbb{S}^1 \times \{0\}$ in $\overline{V}$ represents a generator of $H_1(\overline{V};\mathbb{Z})\cong \mathbb{Z}$ and $[P]$ some integer $q$ times this generator. Define the \emph{absolute winding number of $P$} to be $r(P)=|q|$. The map $K \to P(K)$ is called a \emph{satellite operator}.
\newline
\newline
A \emph{classical satellite} (compare with \cite{sil_wil_sat}) of a virtual knot is defined as follows. Let $\upsilon$ be a virtual knot diagram. Replace each arc of the diagram with $p$ parallel arcs in $\mathbb{R}^2$. At each overcrossing/undercrossing arc of $\upsilon$, the $p$ parallel arcs all pass over/under, respectively. Mark $p^2$ new virtual crossings for each virtual crossing of $\upsilon$. Lastly, break the $p$ strands at some point away from the classical and virtual crossings, and insert an oriented $(p,p)$-tangle $\tau$ such that the result is an oriented knot. \emph{We further require that all the crossings in} $\tau$ \emph{be classical}.  See Figure \ref{fig_class_sat}. The resulting virtual knot diagram $\tau(\upsilon)$ will be called a \emph{classical satellite with companion} $\upsilon$.  Let $r(\tau)$ denote the difference in the number of incoming and outgoing strands from the top of $\tau$, in absolute value. 
\newline
\newline
An untwisted satellite with companion $K$ can clearly be represented as a classical satellite of the classical knot $K$. For links in SSF we have the following similar result.

\begin{lemma} \label{lemm_sat_class} Let $J \sqcup K$ be a two component link with $J$ fibered and $K$ in SSF with respect to some fiber $\Sigma_J$. Let $P$ be a pattern. Then $J \sqcup P(K)$ is in SSF with respect to $\Sigma_J$ and there is a (non-unique) $(p,p)$-tangle $\tau$ such that $r(P)=r(\tau)$ and $\kappa([P(K);\Sigma_J]) \leftrightharpoons \tau(\kappa([K;\Sigma_J]))$.
\end{lemma}
\begin{proof} Obtain a $(p,p)$-tangle $\tau'$ by cutting $\overline{V}$ along a meridianal disc intersecting $P$ transversely. For each of the arcs in the SSF of $K$, both those lying in $\Sigma_J$ and in the embedded $3$-balls $B_i$, draw $p$ parallel arcs. Let $\beta$ be a $(p,p)$-tangle representing a sufficient number of full twists of the $p$ strands to satisfy the requirement that $\text{lk}(K,l)=0$, for a longitude $l$ of $K$. Insert a diagram of $\tau=\tau'\cdot \beta$ on $\Sigma_J$ into a portion away from the crossings of $[K;\Sigma_J]$. The resulting link is $J \sqcup P(K)$. By construction of $\tau$,  $P(K)$ is in SSF with respect to $\Sigma_J$, $r(P)=r(\tau')=r(\tau)$, and $\kappa([P(K);\Sigma_J]) \leftrightharpoons \tau(\kappa([K;\Sigma_J]))$.
\end{proof}

\begin{figure}[htb]
\begin{tabular}{|c|} \hline \\
\def\svgwidth{2in}
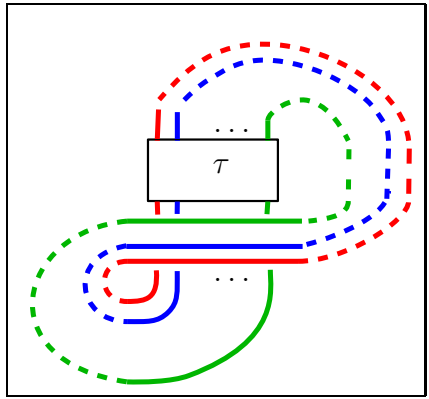 \\ \hline
\end{tabular}
\caption{A schematic of a classical satellite $\tau(\upsilon)$  of a virtual knot $\upsilon$. The $p$ strands are dyed different colors, but all lie one component.} \label{fig_class_sat}
\end{figure}

The following theorem generalizes a result of A. Gibson \cite{gibson_cables} on cables to classical satellites.  It is used in the next section to study injectivity of  satellite operators in sf-concordance.

\begin{theorem} \label{thm_gib_gen} Let $\upsilon$ be a virtual knot diagram and $\tau(\upsilon)$ a classical satellite with companion $\upsilon$. Let $r=r(\tau)$. Then $w_{\tau(\upsilon)}(t)=r^2 w_{\upsilon}(t^r)$.
\end{theorem}
\begin{proof} Consider each of the $p$ parallel strands drawn in the construction of $\tau(\upsilon)$ to be dyed a different color, labeled as $c_1,\ldots,c_p$. If $x$ is a classical crossing of $\upsilon$, let $\tau(x)$ denote the set of $p^2$ corresponding crossings in $\tau(\upsilon)$. Thus, every classical crossing of $\tau(\upsilon)$ is either in $\tau(x)$ for some $x$ or in $\tau$. For each crossing $y$ of $\tau(\upsilon)$, we will compute its index by traversing the knot diagram from $y$ and counting contributions of $\text{int}_y(a)$ until the first time we return to $y$. 
\newline
\newline
Note that an entire dyed strand of color $c_i$ is traversed before returning, then the net contribution of those crossings passed is zero. This is clear for crossings of the strands $c_i$ with itself, since the trip passes both over and under the crossing. Now consider a crossing $z_1 \in \tau(x)$ where $c_i$ crosses $c_j$ and $j \ne i$. Observe that there is exactly one other crossing $z_2 \in \tau(x)$ such that $c_i$ crosses $c_j$. It is easy to see that $\text{int}_y(z_1)=-\text{int}_y(z_2)$ regardless of how the strands happen to be oriented by $\tau$. The net change to $\text{index}(y)$ is thus zero.
\newline
\newline
Consider first a classical crossing $y$ in $\tau$. Then the only contribution to its index comes from other crossing in $\tau$. Thus we may replace all crossings in $\tau(\upsilon)$ outside of $\tau$ with virtual crossings without affecting $\text{index}(y)$. This results in a classical knot, hence $\text{index}(y)=0$.
\newline
\newline
Now consider $y \in \tau(y_0)$ where $y_0$ is a classical crossing of $\upsilon$. For all $w \ne y_0$, convert all crossings in $\tau(w)$ to virtual ones. As detour moves do not affect the index, we may assume that $\tau$ and the crossings of $\tau(y_0)$ are positioned as in Figure \ref{fig_class_sat}. This virtual knot is classical, so the total contributions to $\text{index}(y)$ from crossings in $\tau$ and those in $\tau(y_0)$ is zero. Hence, we count only crossings in $\tau(w)$ with $w \ne y_0$.  Each contribution to $\text{index}(y_0)$ in $\upsilon$ counts $ \pm r$ times in $\text{index}(y)$. Hence, $\text{index}(y)=\pm r \cdot \text{index}(y_0)$. Lastly note that the sum of signs of crossings in $\tau(y_0)$ is $r^2$ (resp. $-r^2$) when $y_0$ is signed $\oplus$ (resp. $\ominus$) .
\end{proof}
\subsection{Injectivity of Satellite Operators} A satellite operator $P$ is said to be \emph{injective} if $K_0 \not \asymp K_1$ implies $P(K_0) \not \asymp P(K_1)$ for all knots $K_0,K_1$ in $\mathbb{S}^3$. In \cite{coch_ray}, the injectivity of $P$ having strong winding number $\pm 1$ was established for the topological, exotic, and smooth categories (assuming the smooth Poincar\'{e} Conjecture on $\mathbb{S}^4$).  
\newline
\newline
Consider the satellite operator $J \sqcup K \to J \sqcup P(K)$. We will say that $P$ is \emph{injective in semi-fibered concordance} if $J_0 \sqcup K_0 \not \asymp_{sf} J_1 \sqcup K_1$ implies $J_0 \sqcup P(K_0) \not \asymp_{sf} J_1 \sqcup P(K_1)$.
\newline
\newline
\textbf{Example:} Embed the standard ribbon disc of the square knot in a solid torus so that it is not contained in a $3$-ball. Let $P$ be the pattern corresponding to this square knot. Let $J \sqcup K$ be the link on the left in Figure \ref{fig_sato_levine_1}. By Theorem \ref{thm_ribbon} (2), $J \sqcup P(K) \asymp_{sf} J \sqcup \bigcirc \asymp_{sf} J \sqcup P(\bigcirc)$, where $\bigcirc$ bounds a disc in the complement of $J$. However, we showed that $J \sqcup K \not\asymp_{sf} J \sqcup \bigcirc$ since the associated virtual  knot has non-vanishing HT polynomial. Thus $P$ is not injective on the set of all semi-fibered concordance classes of  two-component links. \hfill $\square$
\newline
\newline
Two links in SSF whose associated virtual knots have different HT polynomials are not semi-fibered concordant. This condition is sufficient to prove injectivity in semi-fibered concordance of non-zero absolute winding number satellite operators applied to such links.

\begin{theorem} For $i=0,1$, let $J_i \sqcup K_i$ be a two component link with $J_i$ fibered, $K_i$ in SSF with respect to some fiber $\Sigma_i$ of $J_i$, and invariant associated virtual knot $\upsilon_i=\kappa([K_i;\Sigma_i])$. Let $P$ be a pattern with $r(P) \ne 0$. If $w_{\upsilon_0}(t) \ne w_{\upsilon_1}(t)$, then $J_0 \sqcup P(K_0) \not \asymp_{sf} J_1 \sqcup P(K_1)$.
\end{theorem}
\begin{proof} Let $\tau'$ be a $(p,p)$-tangle obtained by cutting the solid torus containing $P$ along a meridianal disc. By Lemma \ref{lemm_sat_class} and its proof, for $i=0,1$ there is a $(p,p)$-tangle $\tau_i=\tau'\cdot \beta_i$ such that $\kappa([P(K_i);\Sigma_i]) \leftrightharpoons \tau_i(\upsilon_i)$, where $\beta_i$ is some appropriate number of full twists for $K_i$. Suppose that $J \sqcup P(K_0) \asymp_{sf} J \sqcup P(K_1)$. By Theorems \ref{theorem_sf_invar} and \ref{thm_ht_invar}, $w_{\tau_0(\upsilon_0)}(t)=w_{\tau_1(\upsilon_1)}(t)$. Since $r(P) \ne 0$, Theorem \ref{thm_gib_gen} implies that $w_{\upsilon_0}(t^{r(P)})=w_{\upsilon_1}(t^{r(P)})$.  This is a contradiction.
\end{proof}

\section{Concordance and Cables of Knots in 3-manifolds}
\label{sec_three}
\subsection{Concordance in $3$-Manifolds} Virtual covers can also be used to study concordance of knots in closed oriented $3$-manifolds $N$. In this section we sketch some applications and examples. We will say that oriented knots $K_0^N$ and $K_1^N$ are concordant in $N$ if there is a properly embedded annulus $A$ in $N \times \mathbb{I}$ such that for $i=0,1$, $A \cap (N \times \{i\})=(-1)^i K_i$. As usual, we write $K_0^N \asymp K_1^N$ to denote concordance in $N$. 

\begin{theorem}\label{thm_conc_in_man} Let $(\mathfrak{k}_0^{\Sigma \times \mathbb{R}},\Pi,K_0^N),(\mathfrak{k}_1^{\Sigma \times \mathbb{R}},\Pi,K_1^N)$ be virtual covers with invariant associated virtual knots $\upsilon_0,\upsilon_1$, respectively. If $K_0^N \asymp K_1^N$, then $\upsilon_0 \asymp \upsilon_1$.
\end{theorem}
\begin{proof} It is similar to the proof of Theorem \ref{theorem_sf_invar}, and hence we leave it as an exercise.
\end{proof}

A $3$-manifold $N$ is said to be \emph{fibered} if it can be represented as a mapping torus $\Sigma\times \mathbb{I}/\psi$, where $\psi:\Sigma \to \Sigma$ is an orientation preserving diffeomorphism. As such it is a fiber bundle over $\mathbb{S}^1$ with fiber $\Sigma$. We will assume $\Sigma$ is c.c.o., so that the covering space $\Pi:\Sigma \times \mathbb{R} \to N$ defined by the mapping torus is regular and orientation preserving.  Similar to knots in fibered knot complements, we can define \emph{special surface form} (SSF, again), where a knot can be decomposed into arcs in a fixed fiber and ``crossings'' in small balls, each intersecting $\Sigma$ is a disc. 
\newline
\newline
When comparing knots in $N$, we will always state the explicit hypothesis that both knots are in SSF with respect to the \emph{same fiber} of a given fibration/mapping torus. Under this condition, as in \cite{cm_fiber}, it follows that a knot in SSF has an invariant associated virtual knot and equivalent knots in SSF with resepect to the same fiber have equivalent invariant associated virtual knots.
\newline
\newline
\textbf{Example:} A handlebody decomposition (from Exercise 8.2.4 of \cite{kirby_calc}) of a $4$-manifold whose boundary is a fibered $3$-manifold $N$, is given in the center picture in Figures \ref{fig_example_1_1} and \ref{fig_example_2_1}. Here $n$ is any integer. A torus fiber can be seen as follows. Let $D$ be the visible disc contained in the plane of the paper, bounded by the four black $0$-framed arcs and arcs intersecting the boundaries of the attaching regions of the $1$-handles. In the $4$-manifold, this is a torus with a disc removed. The $0$-framed arcs, viewed as a $3$-manifold Dehn surgery, attach a disc along the boundary of the removed disc to create a fiber $\Sigma \approx \mathbb{S}^1 \times \mathbb{S}^1$. The $-1/n$-framed arc in the diagram corresponds to the orientation preserving diffeomorphism in the mapping torus $\psi^n:\Sigma \to \Sigma$, where $\psi$ is a Dehn twist (see \cite{kirby_calc}).
\newline
\newline
Let $K_0^N$ be the red knot indicated in the middle of Figure \ref{fig_example_1_1}. Let $K_1^N$ be the red knot indicated in the middle of Figure \ref{fig_example_2_1}. The left hand side of Figure \ref{fig_example_1_1} (resp., Figure \ref{fig_example_2_1}) shows $K_0^N$ (resp., $K_1^N$) in the more conventional from of a \emph{mixed link diagram} \cite{lamb_rourke}. Both $K_0^N$ and $K_1^N$ are in SSF with respect to $\Sigma$. The invariant associated virtual knots $\upsilon_0,\upsilon_1$ are given on the right hand side of their respective figures. Since $w_{\upsilon_0}(t) =2t \ne 0 =w_{\upsilon_1}(t)$, it follows that $K_0^N \not \asymp K_1^N$. Note that these knots represent the same homology class when considered as curves on $\Sigma$. By Theorem \ref{thm_gen_ras}, the slice genus of both $\upsilon_0$ and $\upsilon_1$ is 1, so neither $K_0^N$ nor $K_1^N$ is concordant to the boundary of an embedded disk in $N$.  \hfill $\square$

\begin{figure}[htb]
\begin{tabular}{|ccc|} \hline & & \\
\begin{tabular}{c}
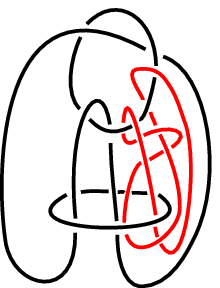
\end{tabular}
& 
\begin{tabular}{c}
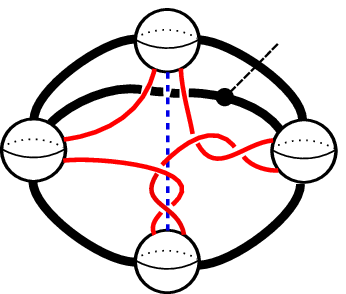 
\end{tabular}
& 
\begin{tabular}{c}
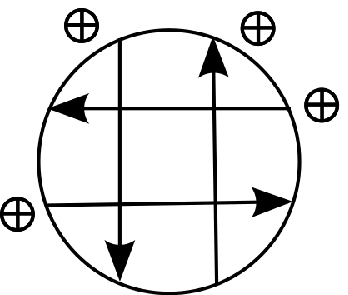 
\end{tabular}
\\ & & \\ \hline
\end{tabular}
\caption{A knot $K_0^N$ in a fibered $3$-manifold $N$ (left). A Gauss diagram of the invariant associated virtual knot $\upsilon_0$ to $K_0^N$ (right).}\label{fig_example_1_1}
\end{figure}

\begin{figure}[htb]
\begin{tabular}{|ccc|} \hline & & \\
\begin{tabular}{c}
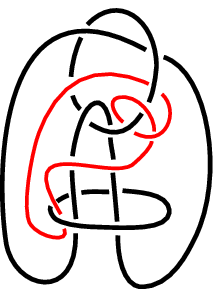
\end{tabular}
& 
\begin{tabular}{c}
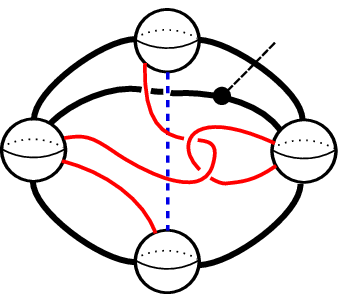 
\end{tabular}
& 
\begin{tabular}{c}
\def\svgwidth{1.35in}
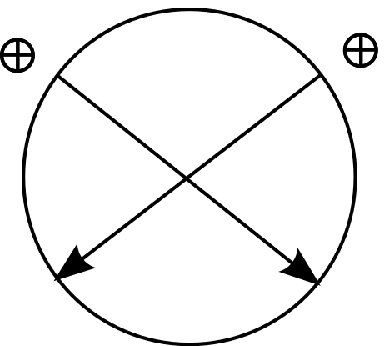 
\end{tabular}
\\ & & \\ \hline
\end{tabular}
\caption{A knot $K_1^N$ in a fibered $3$-manifold $N$ (left).  A Gauss diagram of the invariant associated virtual knot  $\upsilon_1$ to $K_1^N$ (right).}\label{fig_example_2_1}
\end{figure}

\subsection{Cables and Irreducible $3$-Manifolds} A $3$-manifold is said to be \emph{irreducible} if every embedded $\mathbb{S}^2$ bounds an embedded $\mathbb{B}^3$. Otherwise, it is said to be \emph{reducible}. A link $L$ in $\mathbb{S}^3$ is split if and only if $\overline{\mathbb{S}^3\backslash V(L)}$ is reducible. If $L=J \sqcup K$ is a two component split link with $J$ fibered and $K$ in SSF with respect to a fiber $\Sigma_J$, then $K$ is the invariant associated virtual knot \cite{vc_1}. Thus $\overline{\mathbb{S}^3\backslash V(L)}$ is irreducible whenever the invariant associated virtual knot is non-classical. Here  we apply this idea to Dehn surgeries on knots in fibered $3$-manifolds that yield reducible $3$-manifolds.
\newline
\newline
Let $N$ be a closed oriented $3$-manifold and $K^N$ a knot with fixed trivialization of the normal bundle of $K$. Let $P$ be a knot in $\overline{V}=\mathbb{B}^2 \times \mathbb{S}^1$ not contained in an embedded $3$-ball in $\overline{V}$. As in the case of $N=\mathbb{S}^3$, we may define a satellite operator via a longitude $l$ of a closed tubular neighborhood $\overline{V}(K)$ and an o.p. diffeomorphism $f:\overline{V} \to \overline{V}(K)$ identifying $f(\mathbb{S}^1 \times \{0\})$ with $K$, a meridian of $\overline{V}$ to a meridian of $\overline{V}(K)$, and the longitude $\mathbb{S}^1 \times \{1\}$ to $l$. Set $P(K,l)^N:=f(P)^N$.
\newline
\newline
Let $S_{1/2}$ be the circle of radius $1/2$ centered at $0$ in $\mathbb{B}^2$. If $P$ as above is equivalent in $\overline{V}$ to a non-trivial $(p,q)$ torus knot on the torus $S_{1/2} \times \mathbb{S}^1$, then $P(K,l)^N$ is said to be \emph{cabled}. 

\begin{theorem} Let $K^N$ be a knot in a fibered $3$-manifold in SSF with respect to a fiber $\Sigma$ and $\upsilon$ the invariant associated virtual knot. If $\upsilon$ is non-classical and there is a non-trivial Dehn surgery on $K$ that yields a reducible $3$-manifold, then $K$ is cabled in $N$.
\end{theorem}
\begin{proof} Let $p:N \to \mathbb{S}^1$ denote the fiber bundle projection. Scharlemann and Thompson proved \cite{ST_surf} that the reducibility condition implies either: (1) $p(K)$ has non-zero degree, (2) $K$ is contained in an embedded $3$-ball, (3) $K$ is cabled in $N$, or (4) $K$ is a simple closed curve in a fiber. Since $K$ lifts to a knot in $\Sigma \times \mathbb{R}$, option (1) is impossible. If $K$ is contained in an embedded $3$-ball, then $\upsilon$ must be classical. If $K$ is a simple closed curve in a fiber, its invariant associated virtual knot is equivalent to the trivial knot. Thus the hypotheses on $\upsilon$ eliminate all options but (3).
\end{proof}

We remark that the invariant associated virtual knot can be used to eliminate some types of cabling of $K^N$. Suppose for example that $K^N$ is cabled so that the companion $C^N$ is in SSF with respect to a fiber. Then there is a virtual cover $(\mathfrak{k}^{\Sigma \times \mathbb{R}},\Pi,K^N)$ such that $\mathfrak{k}^{\Sigma \times \mathbb{R}}$ is cabled with companion $\mathfrak{c}^{\Sigma \times \mathbb{R}}$, where $(\mathfrak{c}^{\Sigma \times \mathbb{R}},\Pi,C^N)$ is a virtual cover. By Theorem \ref{thm_gib_gen}, the HT polynomial of the invariant associated virtual knot $\upsilon$ of $K^N$ depends only on $\kappa(\mathfrak{c}^{\Sigma \times \mathbb{R}})$  and $r(P)$, where $P$ is the torus knot pattern of $K^N$. Thus $w_{\upsilon}(t)$ can exclude such $C^N$. 

\subsection{Acknowledgments} We would like to thank the the organizers and participants of Mathematisches Forschungsinstut Oberwolfach Workshop 1422a for a most enlightening and productive meeting in Summer 2014. Portions of this work were completed while the first named author was on sabbatical from Monmouth University in Spring 2015. Lastly we would like to express our deep gratitude to R. Todd for many helpful conversations.

\bibliographystyle{plain}
\bibliography{bib_cobordisms}
\end{document}

%% file: birth_death_right.eps_tex
\begingroup%
  \makeatletter%
  \providecommand\color[2][]{%
    \errmessage{(Inkscape) Color is used for the text in Inkscape, but the package 'color.sty' is not loaded}%
    \renewcommand\color[2][]{}%
  }%
  \providecommand\transparent[1]{%
    \errmessage{(Inkscape) Transparency is used (non-zero) for the text in Inkscape, but the package 'transparent.sty' is not loaded}%
    \renewcommand\transparent[1]{}%
  }%
  \providecommand\rotatebox[2]{#2}%
  \ifx\svgwidth\undefined%
    \setlength{\unitlength}{25.20000048bp}%
    \ifx\svgscale\undefined%
      \relax%
    \else%
      \setlength{\unitlength}{\unitlength * \real{\svgscale}}%
    \fi%
  \else%
    \setlength{\unitlength}{\svgwidth}%
  \fi%
  \global\let\svgwidth\undefined%
  \global\let\svgscale\undefined%
  \makeatother%
  \begin{picture}(1,0.9814108)%
    \put(0,0){\includegraphics[width=\unitlength]{birth_death_right.eps}}%
  \end{picture}%
\endgroup%

%% file: saddle_left.eps_tex
\begingroup%
  \makeatletter%
  \providecommand\color[2][]{%
    \errmessage{(Inkscape) Color is used for the text in Inkscape, but the package 'color.sty' is not loaded}%
    \renewcommand\color[2][]{}%
  }%
  \providecommand\transparent[1]{%
    \errmessage{(Inkscape) Transparency is used (non-zero) for the text in Inkscape, but the package 'transparent.sty' is not loaded}%
    \renewcommand\transparent[1]{}%
  }%
  \providecommand\rotatebox[2]{#2}%
  \ifx\svgwidth\undefined%
    \setlength{\unitlength}{72.00000029bp}%
    \ifx\svgscale\undefined%
      \relax%
    \else%
      \setlength{\unitlength}{\unitlength * \real{\svgscale}}%
    \fi%
  \else%
    \setlength{\unitlength}{\svgwidth}%
  \fi%
  \global\let\svgwidth\undefined%
  \global\let\svgscale\undefined%
  \makeatother%
  \begin{picture}(1,0.59292614)%
    \put(0,0){\includegraphics[width=\unitlength]{saddle_left.eps}}%
  \end{picture}%
\endgroup%

%% file: saddle_right.eps_tex
\begingroup%
  \makeatletter%
  \providecommand\color[2][]{%
    \errmessage{(Inkscape) Color is used for the text in Inkscape, but the package 'color.sty' is not loaded}%
    \renewcommand\color[2][]{}%
  }%
  \providecommand\transparent[1]{%
    \errmessage{(Inkscape) Transparency is used (non-zero) for the text in Inkscape, but the package 'transparent.sty' is not loaded}%
    \renewcommand\transparent[1]{}%
  }%
  \providecommand\rotatebox[2]{#2}%
  \ifx\svgwidth\undefined%
    \setlength{\unitlength}{71.999999bp}%
    \ifx\svgscale\undefined%
      \relax%
    \else%
      \setlength{\unitlength}{\unitlength * \real{\svgscale}}%
    \fi%
  \else%
    \setlength{\unitlength}{\svgwidth}%
  \fi%
  \global\let\svgwidth\undefined%
  \global\let\svgscale\undefined%
  \makeatother%
  \begin{picture}(1,0.61849893)%
    \put(0,0){\includegraphics[width=\unitlength]{saddle_right.eps}}%
  \end{picture}%
\endgroup%

%% file: int_1.eps_tex
\begingroup%
  \makeatletter%
  \providecommand\color[2][]{%
    \errmessage{(Inkscape) Color is used for the text in Inkscape, but the package 'color.sty' is not loaded}%
    \renewcommand\color[2][]{}%
  }%
  \providecommand\transparent[1]{%
    \errmessage{(Inkscape) Transparency is used (non-zero) for the text in Inkscape, but the package 'transparent.sty' is not loaded}%
    \renewcommand\transparent[1]{}%
  }%
  \providecommand\rotatebox[2]{#2}%
  \ifx\svgwidth\undefined%
    \setlength{\unitlength}{79.66466303bp}%
    \ifx\svgscale\undefined%
      \relax%
    \else%
      \setlength{\unitlength}{\unitlength * \real{\svgscale}}%
    \fi%
  \else%
    \setlength{\unitlength}{\svgwidth}%
  \fi%
  \global\let\svgwidth\undefined%
  \global\let\svgscale\undefined%
  \makeatother%
  \begin{picture}(1,1.02467941)%
    \put(0,0){\includegraphics[width=\unitlength]{int_1.eps}}%
    \put(0.27195104,0.9488312){\color[rgb]{0,0,0}\makebox(0,0)[lb]{\smash{$x$}}}%
    \put(0.738306,0.4970571){\color[rgb]{0,0,0}\makebox(0,0)[lb]{\smash{$a$}}}%
    \put(-0.00245402,0.02354414){\color[rgb]{0,0,0}\makebox(0,0)[lb]{\smash{$\text{int}_x(a)=1$}}}%
  \end{picture}%
\endgroup%

%% file: int_2.eps_tex
\begingroup%
  \makeatletter%
  \providecommand\color[2][]{%
    \errmessage{(Inkscape) Color is used for the text in Inkscape, but the package 'color.sty' is not loaded}%
    \renewcommand\color[2][]{}%
  }%
  \providecommand\transparent[1]{%
    \errmessage{(Inkscape) Transparency is used (non-zero) for the text in Inkscape, but the package 'transparent.sty' is not loaded}%
    \renewcommand\transparent[1]{}%
  }%
  \providecommand\rotatebox[2]{#2}%
  \ifx\svgwidth\undefined%
    \setlength{\unitlength}{82.53830852bp}%
    \ifx\svgscale\undefined%
      \relax%
    \else%
      \setlength{\unitlength}{\unitlength * \real{\svgscale}}%
    \fi%
  \else%
    \setlength{\unitlength}{\svgwidth}%
  \fi%
  \global\let\svgwidth\undefined%
  \global\let\svgscale\undefined%
  \makeatother%
  \begin{picture}(1,0.98900427)%
    \put(0,0){\includegraphics[width=\unitlength]{int_2.eps}}%
    \put(0.26248282,0.91579679){\color[rgb]{0,0,0}\makebox(0,0)[lb]{\smash{$x$}}}%
    \put(0.71260122,0.47975161){\color[rgb]{0,0,0}\makebox(0,0)[lb]{\smash{$a$}}}%
    \put(-0.00236858,0.02272443){\color[rgb]{0,0,0}\makebox(0,0)[lb]{\smash{$\text{int}_x(a)=-1$}}}%
  \end{picture}%
\endgroup%

%% file: dist_half.eps_tex
\begingroup%
  \makeatletter%
  \providecommand\color[2][]{%
    \errmessage{(Inkscape) Color is used for the text in Inkscape, but the package 'color.sty' is not loaded}%
    \renewcommand\color[2][]{}%
  }%
  \providecommand\transparent[1]{%
    \errmessage{(Inkscape) Transparency is used (non-zero) for the text in Inkscape, but the package 'transparent.sty' is not loaded}%
    \renewcommand\transparent[1]{}%
  }%
  \providecommand\rotatebox[2]{#2}%
  \ifx\svgwidth\undefined%
    \setlength{\unitlength}{126.00000086bp}%
    \ifx\svgscale\undefined%
      \relax%
    \else%
      \setlength{\unitlength}{\unitlength * \real{\svgscale}}%
    \fi%
  \else%
    \setlength{\unitlength}{\svgwidth}%
  \fi%
  \global\let\svgwidth\undefined%
  \global\let\svgscale\undefined%
  \makeatother%
  \begin{picture}(1,0.46051844)%
    \put(0,0){\includegraphics[width=\unitlength]{dist_half.eps}}%
    \put(0.78714591,0.22842345){\color[rgb]{0,0,0}\makebox(0,0)[lb]{\smash{$D_x$}}}%
    \put(0.07487732,0.22842345){\color[rgb]{0,0,0}\makebox(0,0)[lb]{\smash{$D_x'$}}}%
  \end{picture}%
\endgroup%

%% file: virt_cov_ex_am_1.eps_tex
\begingroup%
  \makeatletter%
  \providecommand\color[2][]{%
    \errmessage{(Inkscape) Color is used for the text in Inkscape, but the package 'color.sty' is not loaded}%
    \renewcommand\color[2][]{}%
  }%
  \providecommand\transparent[1]{%
    \errmessage{(Inkscape) Transparency is used (non-zero) for the text in Inkscape, but the package 'transparent.sty' is not loaded}%
    \renewcommand\transparent[1]{}%
  }%
  \providecommand\rotatebox[2]{#2}%
  \ifx\svgwidth\undefined%
    \setlength{\unitlength}{115.48198139bp}%
    \ifx\svgscale\undefined%
      \relax%
    \else%
      \setlength{\unitlength}{\unitlength * \real{\svgscale}}%
    \fi%
  \else%
    \setlength{\unitlength}{\svgwidth}%
  \fi%
  \global\let\svgwidth\undefined%
  \global\let\svgscale\undefined%
  \makeatother%
  \begin{picture}(1,1.09088865)%
    \put(0,0){\includegraphics[width=\unitlength]{virt_cov_ex_am_1.eps}}%
  \end{picture}%
\endgroup%

%% file: virt_cov_ex_am_2.eps_tex
\begingroup%
  \makeatletter%
  \providecommand\color[2][]{%
    \errmessage{(Inkscape) Color is used for the text in Inkscape, but the package 'color.sty' is not loaded}%
    \renewcommand\color[2][]{}%
  }%
  \providecommand\transparent[1]{%
    \errmessage{(Inkscape) Transparency is used (non-zero) for the text in Inkscape, but the package 'transparent.sty' is not loaded}%
    \renewcommand\transparent[1]{}%
  }%
  \providecommand\rotatebox[2]{#2}%
  \ifx\svgwidth\undefined%
    \setlength{\unitlength}{100.96176491bp}%
    \ifx\svgscale\undefined%
      \relax%
    \else%
      \setlength{\unitlength}{\unitlength * \real{\svgscale}}%
    \fi%
  \else%
    \setlength{\unitlength}{\svgwidth}%
  \fi%
  \global\let\svgwidth\undefined%
  \global\let\svgscale\undefined%
  \makeatother%
  \begin{picture}(1,1.28518558)%
    \put(0,0){\includegraphics[width=\unitlength]{virt_cov_ex_am_2.eps}}%
  \end{picture}%
\endgroup%

%% file: disc_band_inf.eps_tex
\begingroup%
  \makeatletter%
  \providecommand\color[2][]{%
    \errmessage{(Inkscape) Color is used for the text in Inkscape, but the package 'color.sty' is not loaded}%
    \renewcommand\color[2][]{}%
  }%
  \providecommand\transparent[1]{%
    \errmessage{(Inkscape) Transparency is used (non-zero) for the text in Inkscape, but the package 'transparent.sty' is not loaded}%
    \renewcommand\transparent[1]{}%
  }%
  \providecommand\rotatebox[2]{#2}%
  \ifx\svgwidth\undefined%
    \setlength{\unitlength}{2.6in}%
    \ifx\svgscale\undefined%
      \relax%
    \else%
      \setlength{\unitlength}{\unitlength * \real{\svgscale}}%
    \fi%
  \else%
    \setlength{\unitlength}{\svgwidth}%
  \fi%
  \global\let\svgwidth\undefined%
  \global\let\svgscale\undefined%
  \makeatother%
  \begin{picture}(1,0.36596552)%
    \put(0,0){\includegraphics[width=\unitlength]{disc_band_inf.eps}}%
    \put(0.1313794,0.21010968){\color[rgb]{0,0,0}\makebox(0,0)[lb]{\smash{}}}%
    \put(0.13452943,0.23058488){\color[rgb]{0,0,0}\makebox(0,0)[lb]{\smash{}}}%
    \put(0.06737553,0.11206587){\color[rgb]{0,0,0}\makebox(0,0)[lb]{\smash{$a_1$ }}}%
    \put(0.32891289,0.12){\color[rgb]{0,0,0}\makebox(0,0)[lb]{\smash{$b_1$}}}%
    \put(0.59330344,0.15105871){\color[rgb]{0,0,0}\makebox(0,0)[lb]{\smash{$a_g$ }}}%
    \put(0.85198764,0.1){\color[rgb]{0,0,0}\makebox(0,0)[lb]{\smash{$b_g$ }}}%
  \end{picture}%
\endgroup%

%% file: inf_surge_1.eps_tex
\begingroup%
  \makeatletter%
  \providecommand\color[2][]{%
    \errmessage{(Inkscape) Color is used for the text in Inkscape, but the package 'color.sty' is not loaded}%
    \renewcommand\color[2][]{}%
  }%
  \providecommand\transparent[1]{%
    \errmessage{(Inkscape) Transparency is used (non-zero) for the text in Inkscape, but the package 'transparent.sty' is not loaded}%
    \renewcommand\transparent[1]{}%
  }%
  \providecommand\rotatebox[2]{#2}%
  \ifx\svgwidth\undefined%
    \setlength{\unitlength}{.75in}%
    \ifx\svgscale\undefined%
      \relax%
    \else%
      \setlength{\unitlength}{\unitlength * \real{\svgscale}}%
    \fi%
  \else%
    \setlength{\unitlength}{\svgwidth}%
  \fi%
  \global\let\svgwidth\undefined%
  \global\let\svgscale\undefined%
  \makeatother%
  \begin{picture}(1,0.40731287)%
    \put(0,0){\includegraphics[width=\unitlength]{inf_surge_1.eps}}%
    \put(0.02757234,0.3){\color[rgb]{0,0,0}\makebox(0,0)[lb]{\smash{$M_i$}}}%
    \put(0.5,0.08){\color[rgb]{0,0,0}\makebox(0,0)[lb]{\smash{$M_{i+1}$}}}%
  \end{picture}%
\endgroup%

%% file: inf_surge_2.eps_tex
\begingroup%
  \makeatletter%
  \providecommand\color[2][]{%
    \errmessage{(Inkscape) Color is used for the text in Inkscape, but the package 'color.sty' is not loaded}%
    \renewcommand\color[2][]{}%
  }%
  \providecommand\transparent[1]{%
    \errmessage{(Inkscape) Transparency is used (non-zero) for the text in Inkscape, but the package 'transparent.sty' is not loaded}%
    \renewcommand\transparent[1]{}%
  }%
  \providecommand\rotatebox[2]{#2}%
  \ifx\svgwidth\undefined%
    \setlength{\unitlength}{.75in}%
    \ifx\svgscale\undefined%
      \relax%
    \else%
      \setlength{\unitlength}{\unitlength * \real{\svgscale}}%
    \fi%
  \else%
    \setlength{\unitlength}{\svgwidth}%
  \fi%
  \global\let\svgwidth\undefined%
  \global\let\svgscale\undefined%
  \makeatother%
  \begin{picture}(1,0.40731287)%
    \put(0,0){\includegraphics[width=\unitlength]{inf_surge_2.eps}}%
    \put(0.02757234,0.3){\color[rgb]{0,0,0}\makebox(0,0)[lb]{\smash{$M_i$}}}%
    \put(0.5,0.08){\color[rgb]{0,0,0}\makebox(0,0)[lb]{\smash{$M_{i+1}$}}}%
  \end{picture}%
\endgroup%

%% file: inf_surge_3.eps_tex
\begingroup%
  \makeatletter%
  \providecommand\color[2][]{%
    \errmessage{(Inkscape) Color is used for the text in Inkscape, but the package 'color.sty' is not loaded}%
    \renewcommand\color[2][]{}%
  }%
  \providecommand\transparent[1]{%
    \errmessage{(Inkscape) Transparency is used (non-zero) for the text in Inkscape, but the package 'transparent.sty' is not loaded}%
    \renewcommand\transparent[1]{}%
  }%
  \providecommand\rotatebox[2]{#2}%
  \ifx\svgwidth\undefined%
    \setlength{\unitlength}{.75in}%
    \ifx\svgscale\undefined%
      \relax%
    \else%
      \setlength{\unitlength}{\unitlength * \real{\svgscale}}%
    \fi%
  \else%
    \setlength{\unitlength}{\svgwidth}%
  \fi%
  \global\let\svgwidth\undefined%
  \global\let\svgscale\undefined%
  \makeatother%
  \begin{picture}(1,0.7016115)%
    \put(0,0){\includegraphics[width=\unitlength]{inf_surge_3.eps}}%
  \end{picture}%
\endgroup%

%% file: inf_surge_4.eps_tex
\begingroup%
  \makeatletter%
  \providecommand\color[2][]{%
    \errmessage{(Inkscape) Color is used for the text in Inkscape, but the package 'color.sty' is not loaded}%
    \renewcommand\color[2][]{}%
  }%
  \providecommand\transparent[1]{%
    \errmessage{(Inkscape) Transparency is used (non-zero) for the text in Inkscape, but the package 'transparent.sty' is not loaded}%
    \renewcommand\transparent[1]{}%
  }%
  \providecommand\rotatebox[2]{#2}%
  \ifx\svgwidth\undefined%
    \setlength{\unitlength}{.75in}%
    \ifx\svgscale\undefined%
      \relax%
    \else%
      \setlength{\unitlength}{\unitlength * \real{\svgscale}}%
    \fi%
  \else%
    \setlength{\unitlength}{\svgwidth}%
  \fi%
  \global\let\svgwidth\undefined%
  \global\let\svgscale\undefined%
  \makeatother%
  \begin{picture}(1,0.69796236)%
    \put(0,0){\includegraphics[width=\unitlength]{inf_surge_4.eps}}%
  \end{picture}%
\endgroup%

%% file: ht_nontriv_inf.eps_tex
\begingroup%
  \makeatletter%
  \providecommand\color[2][]{%
    \errmessage{(Inkscape) Color is used for the text in Inkscape, but the package 'color.sty' is not loaded}%
    \renewcommand\color[2][]{}%
  }%
  \providecommand\transparent[1]{%
    \errmessage{(Inkscape) Transparency is used (non-zero) for the text in Inkscape, but the package 'transparent.sty' is not loaded}%
    \renewcommand\transparent[1]{}%
  }%
  \providecommand\rotatebox[2]{#2}%
  \ifx\svgwidth\undefined%
    \setlength{\unitlength}{108.23404161bp}%
    \ifx\svgscale\undefined%
      \relax%
    \else%
      \setlength{\unitlength}{\unitlength * \real{\svgscale}}%
    \fi%
  \else%
    \setlength{\unitlength}{\svgwidth}%
  \fi%
  \global\let\svgwidth\undefined%
  \global\let\svgscale\undefined%
  \makeatother%
  \begin{picture}(1,0.75230981)%
    \put(0,0){\includegraphics[width=\unitlength]{ht_nontriv_inf.eps}}%
  \end{picture}%
\endgroup%

%% file: ht_nontriv_inf_1.eps_tex
\begingroup%
  \makeatletter%
  \providecommand\color[2][]{%
    \errmessage{(Inkscape) Color is used for the text in Inkscape, but the package 'color.sty' is not loaded}%
    \renewcommand\color[2][]{}%
  }%
  \providecommand\transparent[1]{%
    \errmessage{(Inkscape) Transparency is used (non-zero) for the text in Inkscape, but the package 'transparent.sty' is not loaded}%
    \renewcommand\transparent[1]{}%
  }%
  \providecommand\rotatebox[2]{#2}%
  \ifx\svgwidth\undefined%
    \setlength{\unitlength}{35.7744751bp}%
    \ifx\svgscale\undefined%
      \relax%
    \else%
      \setlength{\unitlength}{\unitlength * \real{\svgscale}}%
    \fi%
  \else%
    \setlength{\unitlength}{\svgwidth}%
  \fi%
  \global\let\svgwidth\undefined%
  \global\let\svgscale\undefined%
  \makeatother%
  \begin{picture}(1,0.92609025)%
    \put(0,0){\includegraphics[width=\unitlength]{ht_nontriv_inf_1.eps}}%
  \end{picture}%
\endgroup%

%% file: ht_nontriv_inf_2.eps_tex
\begingroup%
  \makeatletter%
  \providecommand\color[2][]{%
    \errmessage{(Inkscape) Color is used for the text in Inkscape, but the package 'color.sty' is not loaded}%
    \renewcommand\color[2][]{}%
  }%
  \providecommand\transparent[1]{%
    \errmessage{(Inkscape) Transparency is used (non-zero) for the text in Inkscape, but the package 'transparent.sty' is not loaded}%
    \renewcommand\transparent[1]{}%
  }%
  \providecommand\rotatebox[2]{#2}%
  \ifx\svgwidth\undefined%
    \setlength{\unitlength}{36.46893616bp}%
    \ifx\svgscale\undefined%
      \relax%
    \else%
      \setlength{\unitlength}{\unitlength * \real{\svgscale}}%
    \fi%
  \else%
    \setlength{\unitlength}{\svgwidth}%
  \fi%
  \global\let\svgwidth\undefined%
  \global\let\svgscale\undefined%
  \makeatother%
  \begin{picture}(1,0.92604966)%
    \put(0,0){\includegraphics[width=\unitlength]{ht_nontriv_inf_2.eps}}%
  \end{picture}%
\endgroup%

%% file: ht_nontriv_inf_3.eps_tex
\begingroup%
  \makeatletter%
  \providecommand\color[2][]{%
    \errmessage{(Inkscape) Color is used for the text in Inkscape, but the package 'color.sty' is not loaded}%
    \renewcommand\color[2][]{}%
  }%
  \providecommand\transparent[1]{%
    \errmessage{(Inkscape) Transparency is used (non-zero) for the text in Inkscape, but the package 'transparent.sty' is not loaded}%
    \renewcommand\transparent[1]{}%
  }%
  \providecommand\rotatebox[2]{#2}%
  \ifx\svgwidth\undefined%
    \setlength{\unitlength}{36.46893616bp}%
    \ifx\svgscale\undefined%
      \relax%
    \else%
      \setlength{\unitlength}{\unitlength * \real{\svgscale}}%
    \fi%
  \else%
    \setlength{\unitlength}{\svgwidth}%
  \fi%
  \global\let\svgwidth\undefined%
  \global\let\svgscale\undefined%
  \makeatother%
  \begin{picture}(1,0.92604966)%
    \put(0,0){\includegraphics[width=\unitlength]{ht_nontriv_inf_3.eps}}%
  \end{picture}%
\endgroup%

%% file: ht_nontriv_inf_4.eps_tex
\begingroup%
  \makeatletter%
  \providecommand\color[2][]{%
    \errmessage{(Inkscape) Color is used for the text in Inkscape, but the package 'color.sty' is not loaded}%
    \renewcommand\color[2][]{}%
  }%
  \providecommand\transparent[1]{%
    \errmessage{(Inkscape) Transparency is used (non-zero) for the text in Inkscape, but the package 'transparent.sty' is not loaded}%
    \renewcommand\transparent[1]{}%
  }%
  \providecommand\rotatebox[2]{#2}%
  \ifx\svgwidth\undefined%
    \setlength{\unitlength}{36.46893616bp}%
    \ifx\svgscale\undefined%
      \relax%
    \else%
      \setlength{\unitlength}{\unitlength * \real{\svgscale}}%
    \fi%
  \else%
    \setlength{\unitlength}{\svgwidth}%
  \fi%
  \global\let\svgwidth\undefined%
  \global\let\svgscale\undefined%
  \makeatother%
  \begin{picture}(1,0.92604966)%
    \put(0,0){\includegraphics[width=\unitlength]{ht_nontriv_inf_4.eps}}%
  \end{picture}%
\endgroup%

%% file: ht_nontriv_inf_5.eps_tex
\begingroup%
  \makeatletter%
  \providecommand\color[2][]{%
    \errmessage{(Inkscape) Color is used for the text in Inkscape, but the package 'color.sty' is not loaded}%
    \renewcommand\color[2][]{}%
  }%
  \providecommand\transparent[1]{%
    \errmessage{(Inkscape) Transparency is used (non-zero) for the text in Inkscape, but the package 'transparent.sty' is not loaded}%
    \renewcommand\transparent[1]{}%
  }%
  \providecommand\rotatebox[2]{#2}%
  \ifx\svgwidth\undefined%
    \setlength{\unitlength}{79.0494301bp}%
    \ifx\svgscale\undefined%
      \relax%
    \else%
      \setlength{\unitlength}{\unitlength * \real{\svgscale}}%
    \fi%
  \else%
    \setlength{\unitlength}{\svgwidth}%
  \fi%
  \global\let\svgwidth\undefined%
  \global\let\svgscale\undefined%
  \makeatother%
  \begin{picture}(1,0.97284459)%
    \put(0,0){\includegraphics[width=\unitlength]{ht_nontriv_inf_5.eps}}%
    \put(0.36803493,0.76877115){\color[rgb]{0,0,0}\makebox(0,0)[lb]{\smash{$2k+1$}}}%
  \end{picture}%
\endgroup%

%% file: sato_levine_example_2.eps_tex
\begingroup%
  \makeatletter%
  \providecommand\color[2][]{%
    \errmessage{(Inkscape) Color is used for the text in Inkscape, but the package 'color.sty' is not loaded}%
    \renewcommand\color[2][]{}%
  }%
  \providecommand\transparent[1]{%
    \errmessage{(Inkscape) Transparency is used (non-zero) for the text in Inkscape, but the package 'transparent.sty' is not loaded}%
    \renewcommand\transparent[1]{}%
  }%
  \providecommand\rotatebox[2]{#2}%
  \ifx\svgwidth\undefined%
    \setlength{\unitlength}{144.69879607bp}%
    \ifx\svgscale\undefined%
      \relax%
    \else%
      \setlength{\unitlength}{\unitlength * \real{\svgscale}}%
    \fi%
  \else%
    \setlength{\unitlength}{\svgwidth}%
  \fi%
  \global\let\svgwidth\undefined%
  \global\let\svgscale\undefined%
  \makeatother%
  \begin{picture}(1,0.82337044)%
    \put(0,0){\includegraphics[width=\unitlength]{sato_levine_example_2.eps}}%
  \end{picture}%
\endgroup%

%% file: sato_levine_example_3.eps_tex
\begingroup%
  \makeatletter%
  \providecommand\color[2][]{%
    \errmessage{(Inkscape) Color is used for the text in Inkscape, but the package 'color.sty' is not loaded}%
    \renewcommand\color[2][]{}%
  }%
  \providecommand\transparent[1]{%
    \errmessage{(Inkscape) Transparency is used (non-zero) for the text in Inkscape, but the package 'transparent.sty' is not loaded}%
    \renewcommand\transparent[1]{}%
  }%
  \providecommand\rotatebox[2]{#2}%
  \ifx\svgwidth\undefined%
    \setlength{\unitlength}{144.69879607bp}%
    \ifx\svgscale\undefined%
      \relax%
    \else%
      \setlength{\unitlength}{\unitlength * \real{\svgscale}}%
    \fi%
  \else%
    \setlength{\unitlength}{\svgwidth}%
  \fi%
  \global\let\svgwidth\undefined%
  \global\let\svgscale\undefined%
  \makeatother%
  \begin{picture}(1,0.82337044)%
    \put(0,0){\includegraphics[width=\unitlength]{sato_levine_example_3.eps}}%
  \end{picture}%
\endgroup%

%% file: as_square_surgery_1.eps_tex
\begingroup%
  \makeatletter%
  \providecommand\color[2][]{%
    \errmessage{(Inkscape) Color is used for the text in Inkscape, but the package 'color.sty' is not loaded}%
    \renewcommand\color[2][]{}%
  }%
  \providecommand\transparent[1]{%
    \errmessage{(Inkscape) Transparency is used (non-zero) for the text in Inkscape, but the package 'transparent.sty' is not loaded}%
    \renewcommand\transparent[1]{}%
  }%
  \providecommand\rotatebox[2]{#2}%
  \ifx\svgwidth\undefined%
    \setlength{\unitlength}{137.40143094bp}%
    \ifx\svgscale\undefined%
      \relax%
    \else%
      \setlength{\unitlength}{\unitlength * \real{\svgscale}}%
    \fi%
  \else%
    \setlength{\unitlength}{\svgwidth}%
  \fi%
  \global\let\svgwidth\undefined%
  \global\let\svgscale\undefined%
  \makeatother%
  \begin{picture}(1,0.72703644)%
    \put(0,0){\includegraphics[width=\unitlength]{as_square_surgery_1.eps}}%
    \put(0.61930912,0.45){\color[rgb]{0,0,0}\makebox(0,0)[lb]{\smash{$-1$}}}%
  \end{picture}%
\endgroup%

%% file: as_square_surgery_6.eps_tex
\begingroup%
  \makeatletter%
  \providecommand\color[2][]{%
    \errmessage{(Inkscape) Color is used for the text in Inkscape, but the package 'color.sty' is not loaded}%
    \renewcommand\color[2][]{}%
  }%
  \providecommand\transparent[1]{%
    \errmessage{(Inkscape) Transparency is used (non-zero) for the text in Inkscape, but the package 'transparent.sty' is not loaded}%
    \renewcommand\transparent[1]{}%
  }%
  \providecommand\rotatebox[2]{#2}%
  \ifx\svgwidth\undefined%
    \setlength{\unitlength}{128.3604353bp}%
    \ifx\svgscale\undefined%
      \relax%
    \else%
      \setlength{\unitlength}{\unitlength * \real{\svgscale}}%
    \fi%
  \else%
    \setlength{\unitlength}{\svgwidth}%
  \fi%
  \global\let\svgwidth\undefined%
  \global\let\svgscale\undefined%
  \makeatother%
  \begin{picture}(1,0.84721111)%
    \put(0,0){\includegraphics[width=\unitlength]{as_square_surgery_6.eps}}%
  \end{picture}%
\endgroup%

%% file: ht_sat_proof.eps_tex
\begingroup%
  \makeatletter%
  \providecommand\color[2][]{%
    \errmessage{(Inkscape) Color is used for the text in Inkscape, but the package 'color.sty' is not loaded}%
    \renewcommand\color[2][]{}%
  }%
  \providecommand\transparent[1]{%
    \errmessage{(Inkscape) Transparency is used (non-zero) for the text in Inkscape, but the package 'transparent.sty' is not loaded}%
    \renewcommand\transparent[1]{}%
  }%
  \providecommand\rotatebox[2]{#2}%
  \ifx\svgwidth\undefined%
    \setlength{\unitlength}{303.725bp}%
    \ifx\svgscale\undefined%
      \relax%
    \else%
      \setlength{\unitlength}{\unitlength * \real{\svgscale}}%
    \fi%
  \else%
    \setlength{\unitlength}{\svgwidth}%
  \fi%
  \global\let\svgwidth\undefined%
  \global\let\svgscale\undefined%
  \makeatother%
  \begin{picture}(1,0.89889065)%
    \put(0,0){\includegraphics[width=\unitlength]{ht_sat_proof.eps}}%
    \put(0.47836965,0.55373842){\color[rgb]{0,0,0}\makebox(0,0)[lb]{\smash{$\tau$}}}%
    \put(0.48008818,0.66200917){\color[rgb]{0,0,0}\makebox(0,0)[lb]{\smash{$\ldots$}}}%
    \put(0.48008818,0.27017233){\color[rgb]{0,0,0}\makebox(0,0)[lb]{\smash{$\ldots$}}}%
  \end{picture}%
\endgroup%

%% file: aarons_example_surgery_1_2.eps_tex
\begingroup%
  \makeatletter%
  \providecommand\color[2][]{%
    \errmessage{(Inkscape) Color is used for the text in Inkscape, but the package 'color.sty' is not loaded}%
    \renewcommand\color[2][]{}%
  }%
  \providecommand\transparent[1]{%
    \errmessage{(Inkscape) Transparency is used (non-zero) for the text in Inkscape, but the package 'transparent.sty' is not loaded}%
    \renewcommand\transparent[1]{}%
  }%
  \providecommand\rotatebox[2]{#2}%
  \ifx\svgwidth\undefined%
    \setlength{\unitlength}{60.92419153bp}%
    \ifx\svgscale\undefined%
      \relax%
    \else%
      \setlength{\unitlength}{\unitlength * \real{\svgscale}}%
    \fi%
  \else%
    \setlength{\unitlength}{\svgwidth}%
  \fi%
  \global\let\svgwidth\undefined%
  \global\let\svgscale\undefined%
  \makeatother%
  \begin{picture}(1,1.34462999)%
    \put(0,0){\includegraphics[width=\unitlength]{aarons_example_surgery_1_2.eps}}%
    \put(0.6716201,1.30229216){\color[rgb]{0,0,0}\makebox(0,0)[lb]{\smash{$n$}}}%
    \put(0.41737544,0.08869184){\color[rgb]{0,0,0}\makebox(0,0)[lb]{\smash{$0$}}}%
    \put(0.02381864,1.09680675){\color[rgb]{0,0,0}\makebox(0,0)[lb]{\smash{$0$}}}%
  \end{picture}%
\endgroup%

%% file: aarons_example_surgery_6_III.eps_tex
\begingroup%
  \makeatletter%
  \providecommand\color[2][]{%
    \errmessage{(Inkscape) Color is used for the text in Inkscape, but the package 'color.sty' is not loaded}%
    \renewcommand\color[2][]{}%
  }%
  \providecommand\transparent[1]{%
    \errmessage{(Inkscape) Transparency is used (non-zero) for the text in Inkscape, but the package 'transparent.sty' is not loaded}%
    \renewcommand\transparent[1]{}%
  }%
  \providecommand\rotatebox[2]{#2}%
  \ifx\svgwidth\undefined%
    \setlength{\unitlength}{1.35in}%
    \ifx\svgscale\undefined%
      \relax%
    \else%
      \setlength{\unitlength}{\unitlength * \real{\svgscale}}%
    \fi%
  \else%
    \setlength{\unitlength}{\svgwidth}%
  \fi%
  \global\let\svgwidth\undefined%
  \global\let\svgscale\undefined%
  \makeatother%
  \begin{picture}(1,0.88447791)%
    \put(0,0){\includegraphics[width=\unitlength]{aarons_example_surgery_6_III.eps}}%
    \put(0.12681824,0.75356508){\color[rgb]{0,0,0}\makebox(0,0)[lb]{\smash{$0$}}}%
    \put(0.78429193,0.77426613){\color[rgb]{0,0,0}\makebox(0,0)[lb]{\smash{$-1/n$}}}%
  \end{picture}%
\endgroup%

%% file: gauss_2.eps_tex
\begingroup%
  \makeatletter%
  \providecommand\color[2][]{%
    \errmessage{(Inkscape) Color is used for the text in Inkscape, but the package 'color.sty' is not loaded}%
    \renewcommand\color[2][]{}%
  }%
  \providecommand\transparent[1]{%
    \errmessage{(Inkscape) Transparency is used (non-zero) for the text in Inkscape, but the package 'transparent.sty' is not loaded}%
    \renewcommand\transparent[1]{}%
  }%
  \providecommand\rotatebox[2]{#2}%
  \ifx\svgwidth\undefined%
    \setlength{\unitlength}{97.67136159bp}%
    \ifx\svgscale\undefined%
      \relax%
    \else%
      \setlength{\unitlength}{\unitlength * \real{\svgscale}}%
    \fi%
  \else%
    \setlength{\unitlength}{\svgwidth}%
  \fi%
  \global\let\svgwidth\undefined%
  \global\let\svgscale\undefined%
  \makeatother%
  \begin{picture}(1,0.84574549)%
    \put(0,0){\includegraphics[width=\unitlength]{gauss_2.eps}}%
  \end{picture}%
\endgroup%

%% file: simple_example_surgery_1.eps_tex
\begingroup%
  \makeatletter%
  \providecommand\color[2][]{%
    \errmessage{(Inkscape) Color is used for the text in Inkscape, but the package 'color.sty' is not loaded}%
    \renewcommand\color[2][]{}%
  }%
  \providecommand\transparent[1]{%
    \errmessage{(Inkscape) Transparency is used (non-zero) for the text in Inkscape, but the package 'transparent.sty' is not loaded}%
    \renewcommand\transparent[1]{}%
  }%
  \providecommand\rotatebox[2]{#2}%
  \ifx\svgwidth\undefined%
    \setlength{\unitlength}{61.20000091bp}%
    \ifx\svgscale\undefined%
      \relax%
    \else%
      \setlength{\unitlength}{\unitlength * \real{\svgscale}}%
    \fi%
  \else%
    \setlength{\unitlength}{\svgwidth}%
  \fi%
  \global\let\svgwidth\undefined%
  \global\let\svgscale\undefined%
  \makeatother%
  \begin{picture}(1,1.34080148)%
    \put(0,0){\includegraphics[width=\unitlength]{simple_example_surgery_1.eps}}%
    \put(0.67085426,1.29865445){\color[rgb]{0,0,0}\makebox(0,0)[lb]{\smash{$n$}}}%
    \put(0.41414699,0.14503088){\color[rgb]{0,0,0}\makebox(0,0)[lb]{\smash{$0$}}}%
    \put(0.02597224,1.0940951){\color[rgb]{0,0,0}\makebox(0,0)[lb]{\smash{$0$}}}%
  \end{picture}%
\endgroup%

%% file: simple_example_surgery_6_II.eps_tex
\begingroup%
  \makeatletter%
  \providecommand\color[2][]{%
    \errmessage{(Inkscape) Color is used for the text in Inkscape, but the package 'color.sty' is not loaded}%
    \renewcommand\color[2][]{}%
  }%
  \providecommand\transparent[1]{%
    \errmessage{(Inkscape) Transparency is used (non-zero) for the text in Inkscape, but the package 'transparent.sty' is not loaded}%
    \renewcommand\transparent[1]{}%
  }%
  \providecommand\rotatebox[2]{#2}%
  \ifx\svgwidth\undefined%
    \setlength{\unitlength}{1.35in}%
    \ifx\svgscale\undefined%
      \relax%
    \else%
      \setlength{\unitlength}{\unitlength * \real{\svgscale}}%
    \fi%
  \else%
    \setlength{\unitlength}{\svgwidth}%
  \fi%
  \global\let\svgwidth\undefined%
  \global\let\svgscale\undefined%
  \makeatother%
  \begin{picture}(1,0.8468628)%
    \put(0,0){\includegraphics[width=\unitlength]{simple_example_surgery_6_II.eps}}%
    \put(0.12681824,0.75356508){\color[rgb]{0,0,0}\makebox(0,0)[lb]{\smash{$0$}}}%
    \put(0.78429193,0.77426613){\color[rgb]{0,0,0}\makebox(0,0)[lb]{\smash{$-1/n$}}}%
  \end{picture}%
\endgroup%

%% file: gauss_1.eps_tex
\begingroup%
  \makeatletter%
  \providecommand\color[2][]{%
    \errmessage{(Inkscape) Color is used for the text in Inkscape, but the package 'color.sty' is not loaded}%
    \renewcommand\color[2][]{}%
  }%
  \providecommand\transparent[1]{%
    \errmessage{(Inkscape) Transparency is used (non-zero) for the text in Inkscape, but the package 'transparent.sty' is not loaded}%
    \renewcommand\transparent[1]{}%
  }%
  \providecommand\rotatebox[2]{#2}%
  \ifx\svgwidth\undefined%
    \setlength{\unitlength}{108.85535057bp}%
    \ifx\svgscale\undefined%
      \relax%
    \else%
      \setlength{\unitlength}{\unitlength * \real{\svgscale}}%
    \fi%
  \else%
    \setlength{\unitlength}{\svgwidth}%
  \fi%
  \global\let\svgwidth\undefined%
  \global\let\svgscale\undefined%
  \makeatother%
  \begin{picture}(1,0.89541497)%
    \put(0,0){\includegraphics[width=\unitlength]{gauss_1.eps}}%
  \end{picture}%
\endgroup%